\documentclass[fleqn,11pt]{article}
\usepackage{graphics}
\usepackage{color}
\usepackage{mathrsfs}
\usepackage{amssymb}
\usepackage{amsmath}
\usepackage{boxedminipage}
\usepackage{stmaryrd}
\usepackage{multirow}
\usepackage{booktabs}
\usepackage{accents}
\begin{document}

\pagenumbering{arabic}
\newcounter{comp1}

\newtheorem{definition}{Definition}[section]
\newtheorem{proposition}{Proposition}[section]
\newtheorem{example}{Example}[section]
\newtheorem{method}{Method}[section]
\newtheorem{lemma}{Lemma}[section]
\newtheorem{theorem}{Theorem}[section]
\newtheorem{corollary}{Corollary}[section]
\newtheorem{assumption}{Assumption}
\newtheorem{algorithm}{Algorithm}[section]
\newtheorem{remark}{Remark}[section]
\newcommand{\fig}[1]{\begin{figure}[hbt]
                  \vspace{1cm}
                  \begin{center}
                  \begin{picture}(15,10)(0,0)
                  \put(0,0){\line(1,0){15}}
                  \put(0,0){\line(0,1){10}}
                  \put(15,0){\line(0,1){10}}
                  \put(0,10){\line(1,0){15}}
                  \end{picture}
                  \end{center}
                  \vspace{.3cm}
                  \caption{#1}
                  \vspace{.5cm}
                  \end{figure}}
\newcommand{\Axk}{A(x^k)}
\newcommand{\Aumb}{\sum_{i=1}^{N}A_{i}u_{i}-b}
\newcommand{\Kk}{K^k}
\newcommand{\Kki}{K_{i}^{k}}
\newcommand{\Aukmb}{\sum_{i=1}^{N}A_{i}u_{i}^{k}-b}
\newcommand{\Au}{\sum_{i=1}^{N}A_{i}u_{i}}
\newcommand{\Aukpmb}{\sum_{i=1}^{N}A_{i}u_{i}^{k+1}-b}
\newcommand{\nab}{\nabla^2 f(x^k)}
\newcommand{\xk}{x^k}
\newcommand{\ubk}{\overline{u}^k}
\newcommand{\uhk}{\hat u^k}
\def\QEDclosed{\mbox{\rule[0pt]{1.3ex}{1.3ex}}} % ¶¨ÒåʵÐÄ·û
\def\QEDopen{{\setlength{\fboxsep}{0pt}\setlength{\fboxrule}{0.2pt}\fbox{\rule[0pt]{0pt}{1.3ex}\rule[0pt]{1.3ex}{0pt}}}}
% ¶¨Òå¿ÕÐÄ·û
\def\QED{\QEDopen} % Ñ¡Ìî\QEDclosed µÃµ½ÊµÐÄ
\def\proof{\par\noindent{\em Proof.}}
\def\endproof{\hfill $\Box$ \vskip 0.4cm}
\newcommand{\RR}{\mathbf R}

%\title {\bf
%Nonconvex First-order Primal-dual Method for Nonlinearly Constrained Nonsmooth and Nonconvex Optimization (I)}

\title{A First-Order Primal-Dual Method for Nonconvex Constrained Optimization Based On the Augmented Lagrangian} %-($\Theta(u)+Bv$)}

\author{Daoli Zhu\thanks {Antai College of Economics and Management and Sino-US Global Logistics
Institute, Shanghai Jiao Tong University, 200030 Shanghai, China
({\tt dlzhu@sjtu.edu.cn}).}
\and Lei Zhao\thanks {School of Naval Architecture, Ocean and Civil Engineering, Shanghai
Jiao Tong University, 200030 Shanghai, China ({\tt l.zhao@sjtu.edu.cn}).}
\and Shuzhong Zhang\thanks {Department of Industrial and Systems Engineering, University of Minnesota, Minneapolis, MN 55455, USA
({\tt zhangs@umn.edu}); joint appointment with
School of Data Science, Shenzhen Research Institute of Big Data, The Chinese University of Hong
Kong, Shenzhen, China ({\tt zhangs@cuhk.edu.cn}).}}
\footnotetext[1]{Acknowledgment: this research was supported by NSFC grants 71471112 and 71871140.}

\maketitle

\begin{abstract}
\vspace{1cm}
Nonlinearly constrained nonconvex and nonsmooth optimization models play an increasingly important role in machine learning, statistics and data analytics.
In this paper, based on the augmented Lagrangian function we introduce a flexible first-order primal-dual method, to be called {\em nonconvex auxiliary problem principle of augmented Lagrangian}\/ (NAPP-AL),
for solving a class of nonlinearly constrained nonconvex and nonsmooth optimization problems.
%the nonlinearly constrained nonconvex and nonsmooth problem.
%a class of such problems.
%Each iteration of NAPP-AL generates an approximation for the dual problem based on the augmented Lagrangian function.
%The approximation incorporates both linearization and a distance-like proximal term. Then, the approximation provides a decomposition property for the original problem.
We demonstrate that NAPP-AL converges to a stationary solution at the rate of $o(1/\sqrt{k})$, where $k$ is the number of iterations.
%We establish the $o(1/k)$ running best convergence rate of NAPP-AL.
%Additionally, we propose a value proximity error bound (VP-EB) and derive the linear convergence of NC-APP-AL under this condition.
Moreover, under an additional error bound condition (to be called VP-EB in the paper), we further show that the convergence rate is in fact linear. Finally, we show that the famous Kurdyka-{\L}ojasiewicz property and the metric subregularity imply the afore-mentioned VP-EB condition.

\bigskip

\noindent {\bf Keywords:} nonlinearly constrained nonconvex and nonsmooth optimization, first-order method, primal-dual method, augmented Lagrangian function.

%\vspace{0.5cm}
%
%\noindent {\bf Mathematics Subject Classification:}

\end{abstract}
%La m\'{e}thode du probl\`{e}me auxiliaire permet d'obtenir la solution
%de probl\`{e}mes de minimisation par la r\'{e}solution d'une suite
%de probl\`{e}mes plus simples.
%Nous g\'{e}n\'{e}ralisons cette m\'{e}thode en la couplant
%avec une recherche unidimensionnelle (``linesearch'').
%Nous proposons un algorithme fondamental et \'{e}tudions
%ses propri\'{e}t\'{e}s de convergence
%pour la minimisation de fonctionnelles pseudoconvexes
%en dimension infinie. Nous introduisons \'{e}galement la m\'{e}thode
%du probl\`{e}me auxiliaire partiel de descente qui ne
%lin\'{e}arise que l'objectif et n'introduit qu'un
%sous-ensemble des variables dans le terme auxiliaire.

\normalsize

\newpage

%\vspace{1cm}
\section{Introduction}\label{intro}

\subsection{Nonlinearly constrained nonconvex and nonsmooth optimization}
%\indent Many problems arising in the fields of machine learning, signal and image processing involve solving
In this paper we are concerned with solving the following nonlinearly constrained nonconvex and nonsmooth problem:
%\begin{equation} \label{Prob:general-function}
\[
\begin{array}{lll}
\mbox{(P)}   &\min       & G(u,v)+J(u)+H(v)\\
             &\rm {s.t}  & \Omega(u)+\Phi(u)+Bv=0 \\
             &           & u\in \mathbf{U},v\in\RR^{d}
\end{array}
\]
%\end{equation}
where $G:\RR^{n}\times\RR^{d}\rightarrow\RR$ is a differentiable function, possibly nonconvex; $J:\RR^{n}\rightarrow\RR$ is lower semi-continuous (l.s.c.), possibly nonsmooth and nonconvex; $H: \RR^{d}\rightarrow\RR$ is a differentiable function, possibly nonconvex; $\mathbf{U}$ is closed convex subset of $\RR^{n}$. For simplicity in referencing, let us denote the overall objective to be $F(u,v):=G(u,v)+J(u)+H(v)$, and $\Theta(u):=\Omega(u)+\Phi(u)$, with $\Theta:\RR^{n}\rightarrow\RR^m$ being a differentiable mapping, possibly nonlinear; $B$ is an $m\times d$ tall matrix.

\indent In our bid to solve (P), %\eqref{Prob:general-function},
we keep in mind that variable $u$ in model (P) often has a block structure depicted as follows (where the acronym `b' stands for `block'):
 %to  one consider the following structured problem with multiple block variables:
%\begin{equation}\label{Prob:general-function}
\[
\begin{array}{lll}
\mbox{(P(b))}     & \min       &    G(u_1,\cdots,u_N,v)+\sum\limits_{i=1}^NJ_i(u_i)+H(v)\\
                  & \rm {s.t}  &    \Omega(u_1,\cdots,u_N)+\sum\limits_{i=1}^N\Phi_i(u_i)+Bv=0 \\
                  &            &    u_i\in \mathbf{U}_i, i=1,...,N, v\in\RR^{d}
\end{array}
\]
%\end{equation}
where $J_i$'s are l.s.c.\ but possibly nonsmooth and nonconvex; $\mathbf{U}_i$'s are closed convex subset of $\RR^{n_i}$; $\Phi_i:\RR^{n_i}\rightarrow\RR^m$ is differentiable and possibly nonconvex, and $\sum\limits_{i=1}^N n_i=n$. As we will see later, our newly proposed method is particularly suitable for (P(b)), because in that case the auxiliary subproblems enjoy an advantage of being computable in parallel.

\subsection{Motivating examples}
%In terms of motivation, (P) and (P$_1$) are important and challenging application problems. We present several example applications to demonstrate the reasons for the interest in these problems.
There are numerous applications that can be modelled by (P) or (P(b)). Before discussing the solution method, let us first consider a few illustrative examples below.

\subsubsection{Nonconvex empirical risk minimization}

Empirical risk minimization (ERM) is a popular supervised learning method that is widely used for  classification and regression. The ERM problem can be expressed as follows~\cite{MeiBai2018}:
%\begin{equation}\label{Prob:NERM}
\[
\mbox{(ERM)}\quad \min_{u\in\RR^n} \,\, \phi\left(g(u)\right) + R(u)=\frac{1}{m}\sum_{j=1}^m\phi_j\left(g_j(u)\right)+R(u)
\]
%\end{equation}
where $g:\RR^n\rightarrow\RR^m$ is a smooth mapping (vector-valued function), possibly nonconvex; $\phi_j:\RR\rightarrow\RR$ is a loss function (possibly nonconvex) associated with $g_j(\cdot)$, $j=1,...,m$; and $R:\RR^n\rightarrow\RR$ is a regularization function (possibly nonconvex) for the predictor $u\in\RR^n$. By introducing an auxiliary variable $v\in\RR^m$, the ERM problem can be formulated as follows:
%\begin{equation}\label{prob:erm'}
\[
\begin{array}{ll}
\min  & \phi(v) + R(u)=\frac{1}{m}\sum\limits_{j=1}^m\phi_j(v_j) + R(u)\\
\rm {s.t}  & g_j(u)-v_j=0,\quad j=1,...,m\\
           & u\in \RR^{n},v\in\RR^{m}.
\end{array}
\]
%\end{equation}

\subsubsection{Robust principal component analysis}
Robust principal component analysis (RPCA) is a fundamental tool in machine learning and data science to obtain a low-dimensional expression for high-dimensional data with gross errors. The purpose of RPCA is to decompose a given data matrix $M\in\RR^{m\times n}$ into two parts $M=U^0(V^0)^{\top}+S^0$ where $U^0(V^0)^{\top}$ is a low rank matrix and $S^0$ is a sparse matrix,  %The RPCA problem Robust PCA
formulated as follows~\cite{MaAybat2018}:
%\begin{equation}
\[
\begin{array}{lll}
\mbox{(RPCA)}  &\min  & \|U\|_F^{2}+\|V\|_F^2 +\rho_1 \mathcal{R}(S)+\rho_2 \|N\|_F^2\\
             &\rm {s.t}  & UV^{\top}+S+N-M=0\\
             &           & S, N\in \RR^{m\times n}, U\in\RR^{m\times r}, V\in\RR^{n\times r}
\end{array}
\]
%\end{equation}
where $\rho_1,\rho_2$ are some weight parameters; $\mathcal{R}:\RR^{m\times n}\rightarrow\RR$ is a regularization function (possibly nonconvex) for the sparse matrix $S$. Moreover, $r<\min\{m,n\}$ is the estimated rank of $UV^{\top}$, $N$ is the noise matrix, and $\|\cdot\|_F$ is the Frobenius norm.

\subsubsection{Nonconvex sharing problem}
Consider the following sharing problem~\cite{Boyd2011,HongLuoR2016}:
%\begin{equation}\label{Prob:NSP}
\[
\begin{array}{ll}
\min       & G\left(\sum\limits_{i=1}^N\Theta_i(u_i)\right)+\sum\limits_{i=1}^NJ_i(u_i)\\
\rm {s.t}  & u_i\in \mathbf{U}_i, i=1,...,N
\end{array}
\]
%\end{equation}
where $u_i \in \RR^{n_i}$ is the variable associated with a given agent $i$.
%and $\Theta_i:\RR^{n_i} \rightarrow \RR^m$. The variables are coupled through the function $G(\cdot)$.
To facilitate distributed computation, this problem can be equivalently formulated as a nonlinearly constrained problem by introducing an additional variable $v\in\RR^m$:
%\begin{equation}\label{Prob:NSP'}
\[
\begin{array}{ll}
\min       & G(v)+\sum\limits_{i=1}^NJ_i(u_i) \\
\rm {s.t.} & \sum\limits_{i=1}^N\Theta_i(u_i)-v=0 \\
           & u_i\in \mathbf{U}_i, i=1,...,N.
\end{array}
\]
%\end{equation}

\subsection{Background and related works}

Our bid to solve (P) is based on the notion of {\em Auxiliary Problem Principle}\/ (APP) as introduced by Cohen~\cite{C80}, which is later specialized to constrained optimization via the {\em augmented Lagrangian}\/ function (APP-AL) by Cohen and Zhu~\cite{CohenZ}. The high level concept of APP is to exploit the structure of a computational task at hand by decomposing it into a series of simpler tasks -- the so-called auxiliary problems. Clearly, the approach will need be tailored to the structure of the task in question. In the case of APP-AL, it is essentially a Lagrangian primal-dual approach. In a series of recent papers, Zhao and Zhu~\cite{ZZ20} and Zhao {\em et al.}~\cite{ZZJ17} extended APP-AL to accommodate constrained large-scaled convex optimization models. At the core of APP-AL, the Lagrangian dual variable plays an important role; the main benefit of the approach is to turn a large scale problem into a series of {\em decomposed}\/ simpler subproblems to be solved {\em in parallel}. As a primal-dual method, the philosophy of APP-AL is closely related to another approach that has stirred up much research activities in the recent years: {\em Alternating Direction Method of Multipliers}\/ (ADMM). For a historical account as well as recent extensions of the ADMM, we refer the readers to the excellent survey on ADMM by Boyd {\em et al.}~\cite{Boyd2011}, and the references therein. Though ADMM and APP-AL share common features on the ground of Lagrangian dual-variable gradient updating and primal block-variable decomposition,
the guiding principles of these two approaches differ significantly. For instance, APP-AL aims for primal gradient-proximal updating instead of block optimization; APP-AL advocates for linearization of the augmented Lagrangian function, which makes the Jacobian update approach particularly suitable for distributive computations. %ADMM and APP-AL share a good deal of commonality.
ADMM and its variants apply a well known Gauss-Seidel-like minimization at least for two blocks of primal variables. Noticeable differentiations aside, the boundaries between the two are indeed quite blur.
For instance, under the framework of ADMM, Deng {\em et al.}~\cite{Deng2017} introduced a proximal term in the updates of the primal variables, which can be used to ensure convergence even in the Jacobian-style updates of the multi-block primal variables; Gao and Zhang~\cite{GZ16} introduced a gradient-proximal updating scheme for primal blocks in a composite form under ADMM setting, and proved convergence for convex problems; Jiang {\em et al.}~\cite{JLMZ19}, Wang {\em et al.}~\cite{WYZ19} and Zhang and Luo~\cite{ZL20}  extended the ADMM scheme to nonconvex settings, albeit the constraints are all linear.
Under a broad scheme of augmented Lagrangian approach, Bolte {\em et al.}~\cite{BST18} considered nonconvex and nonsmooth composite optimization, with global convergence guarantees.

\subsection{Main contributions and outline of the paper}
In this paper, we generalize APP-AL~\cite{CohenZ} to solve a {\it nonlinearly and nonconvex}\/ constrained nonsmooth optimization model,
which we shall term {\it nonconvex auxiliary problem principle of augmented Lagrangian}\/ (NAPP-AL) in the paper. At each iteration, NAPP-AL generates a nonlinear approximation of the primal problem of the augmented Lagrangian dual problem. The approximation incorporates both linearization and a Bregman distance-like proximal term. The approximation enables  decomposition to evaluate the primal problem. If appropriate Bregman distance-like function is chosen, then NAPP-AL allows for parallel computation.
In this paper, (i) we show the global convergence of the iterates to a stationary solution, under standard assumptions; (ii) we establish an iteration complexity bound of reaching an $\epsilon$-stationary solution of (P) in $o(1/\epsilon^2)$ iterations; (iii) under an additional error-bound assumption, a linear convergence rate is guaranteed. The paper is organized as follows. In Section~\ref{Pre}, we shall introduce the notations and assumptions surrounding our discussion. In Section~\ref{FBS} we shall present the new solution method for (P). An iteration complexity analysis will be presented in Section~\ref{convergence_FBS}. Finally, in Section~\ref{LKL} we shall show that the convergence rate can be improved to linear, under an error bound condition. Discussions on the nature of the error bound condition and its relations to other properties are presented in Section~\ref{LMS}.

% for the structure problem (P$_1$). The main contributions of this work are the following:
%\begin{itemize}
%\item[\rm(i)] {\bf (The NC-APP-AL algorithm and convergence)} We propose an NC-APP-AL algorithm to solve the nonlinearly constrained nonconvex and nonsmooth optimization problem. We demonstrate that the NC-APP-AL algorithm converges to the set of stationary solutions.
%\item[\rm(ii)] {\bf(Convergence rate)} We establish the $o(1/k)$ running best convergence rate of NC-APP-AL.
%\item[\rm(iii)] {\bf(Linear convergence)} We propose a value proximity error bound (VP-EB); the linear convergence is provided under this condition. We also show that the Kurdyka-{\L}ojasiewicz (K{\L}) property and basic metric subregularity are sufficient conditions for this error bound.
%\end{itemize}

\section{Notations and Assumptions}\label{Pre}
This section provides some useful preliminaries for subsequent discussions and summarizes the notations and assumptions. We denote Euclidean scalar product of $\RR^n$ and Euclidean norm as $\langle\cdot\rangle$ and $\|\cdot\|$, respectively. For a matrix $A$, the minimum eigenvalue is denoted by $\lambda_{\min}(A)$. The spectral norm of a matrix $A$ is denoted by $\|A\|$. Let $\mathbf{C}$ be a subset of $\RR^n$ and $x$ be any point in $\RR^n$. Define
$${\rm dist}(x,\mathbf{C})=\inf\{\|x-z\|:z\in\mathbf{C}\}.$$
By default, we assume ${\rm dist}(x,\mathbf{C})=+\infty$ if $\mathbf{C}=\emptyset$.
We introduce below the definitions of subdifferential calculus and limiting normal cone; see e.g.~\cite{Mordukhovich2006,RockWets2009}.
\begin{definition}[\cite{Mordukhovich2006,RockWets2009}]
Let $\psi$: $\RR^n\rightarrow\RR\cup\{+\infty\}$ be a proper lower semicontinuous function.
\begin{itemize}
\item[{\rm(i)}] The domain of $\psi$ -- denoted by ${\rm dom}(\psi)$ -- is $\{x\in\RR^n:\psi(x)<+\infty\}$.
\item[{\rm(ii)}] For each $\bar{x}\in {\rm dom}(\psi)$, the Fr\'{e}chet subdifferential of $\psi$ at $\bar{x}$ -- denoted by $\partial_{F}\psi(\bar{x})$ -- is the set of all vector $\xi\in\RR^n$ satisfying
    $\lim\limits_{\substack{x\neq\bar{x}\\ x\rightarrow\bar{x}}}\inf\frac{1}{\|x-\bar{x}\|}\left[\psi(x)-\psi(\bar{x})-\langle\xi,x-\bar{x}\rangle\right]\geq0.$
    If $\bar{x}\notin {\rm dom}(\psi)$, then $\partial_F\psi(\bar x)=\emptyset$.
\item[{\rm(iii)}] The limiting-subdifferential (\cite{Mordukhovich2006}), or simply the subdifferential for short, of $\psi$ at $\bar{x}\in {\rm dom}(\psi)$ -- denoted by $\partial\psi(\bar{x})$ -- is defined as:
    $\partial\psi(\bar{x}):=\{\xi\in\RR^n:\exists x_n\rightarrow\bar{x},\, \psi(x_n)\rightarrow \psi(\bar{x}),\, \xi_n\in\partial_F\psi(x_n)\rightarrow\xi\}.$
\item[{\rm(iv)}] The limiting normal cone at $u$ respect to convex set $\mathbf{U}$ is
$\mathcal{N}_{\mathbf{U}}(u)=\{\xi: \: \langle\xi,\zeta-u\rangle\leq0, \forall \zeta\in\mathbf{U}\}.$
Noted that $\mathcal{N}_{\mathbf{U}}(u)=\partial\mathcal{I}_{\mathbf{U}}(u)$, where $\mathcal{I}_{\mathbf{U}}(u)$ is indicator function of set $\mathbf{U}$:
$$
\mathcal{I}_{\mathbf{U}}(u)=\left\{\begin{array}{ll}
0,       & \mbox{if $u\in\mathbf{U}$;}  \\
+\infty, & \mbox{otherwise.}
\end{array}\right.
$$
\end{itemize}
\end{definition}

For convenience, some useful properties of the l.s.c.\ functions are listed in the following proposition.
\begin{proposition}[\cite{Mordukhovich2006,RockWets2009}]
Let $\varphi:\RR^n\rightarrow\RR\cup\{+\infty\}$ and $\psi:\RR^n\rightarrow\RR\cup\{+\infty\}$ be proper l.s.c.\ functions. Then it holds that
\begin{itemize}
\item [{\rm(i)}] If $\varphi$ is continuously differentiable in $u$, then $\partial(\varphi+\psi)(u)=\nabla\varphi(u)+\partial\psi(u)$.
\item [{\rm(ii)}] If $\varphi$ is locally Lipschitz continuous at $u$, $\psi$ is l.s.c.\ and finite at $u$, then $\partial(\varphi+\psi)(u)\subseteq\partial\varphi(u)+\partial\psi(u)$.
\item [{\rm(iii)}] If $\bar{u}$ is a local minimum of $\varphi$, then $0\in\partial\varphi(\bar{u})$.
\end{itemize}
\end{proposition}
The Lagrangian of (P) is defined as
%\begin{equation}\label{func:L}
\[
\mathcal{L}(w)=F(u,v)+\langle p,\Theta(u)+Bv\rangle,\quad\mbox{with}\quad w=(u,v,p)
\]
%\end{equation}
and the augmented Lagrangian function of (P) is defined as
%\begin{equation}\label{func:L}
\[
\mathcal{L}_{\gamma}(w)=F(u,v)+\langle p,\Theta(u)+Bv\rangle+\frac{\gamma}{2}\|\Theta(u)+Bv\|^2.
\]
%\end{equation}
The stationary point of (P) $w\in\mathbf{U}\times\RR^{d}\times\RR^m$ satisfies the following KKT condition:
%\begin{equation}
\[
0\in\partial\mathcal{L}(w)=\left(\begin{array}{c}
\nabla_uG(u,v)+\partial J(u)+\mathcal{N}_{\mathbf{U}}(u)+\left(\nabla\Theta(u)\right)^{\top}p  \\
\nabla_vG(u,v)+\nabla H(v)+B^{\top}p\\
\Theta(u)+Bv
\end{array}\right).
\]
%\end{equation}
We note that this condition is equivalent to
%\begin{equation}
\[
0\in\partial\mathcal{L}_{\gamma}(w) %=\partial\mathcal{L}_{\gamma}(u,v,p)
=\left(\begin{array}{c}
\nabla_uG(u,v)+\partial J(u)+\mathcal{N}_{\mathbf{U}}(u)+\left(\nabla\Theta(u)\right)^{\top}[p+\gamma(\Theta(u)+Bv)]  \\
\nabla_vG(u,v)+\nabla H(v)+B^{\top}[p+\gamma(\Theta(u)+Bv)]\\
\Theta(u)+Bv\end{array}\right),
\]
%\end{equation}
where $\partial J(u)$ is the sub-differential of $J$ at $u$, $\mathcal{N}_{\mathbf{U}}(u)$ is the limiting normal cone to convex set $\mathbf{U}$. In the following, we introduce the notion of {\it approximate}\/ stationary solutions for (P).
\begin{definition} [$\varepsilon$-stationarity for (P)]
We call $w^*\in\mathbf{U}\times\RR^{d}\times\RR^m$ to be a $\varepsilon$-stationary point of (P) if
$${\rm dist}\left(0,\partial\mathcal{L}(w^*)\right)\leq\varepsilon\quad\mbox{(or ${\rm dist}\left(0,\partial\mathcal{L}_{\gamma}(w^*)\right)\leq\varepsilon$)}.$$
\end{definition}
If $\varepsilon=0$, then the above condition implies that $0\in\partial\mathcal{L}(w^*)\quad\mbox{(or $0\in\partial\mathcal{L}_{\gamma}(w^*)$)}$.
Thus, in that case $w^*$ is a stationary point of (P).
Throughout this paper, we make the following rather standard assumptions on model (P) under consideration. %for Problem (P).
\begin{assumption}\label{assump1}
{\ }

\begin{itemize}
\item[{\rm ($H_1$)}] $G:\RR^{n}\times\RR^{d}\rightarrow(-\infty,+\infty]$ is a nonconvex differentiable function with ${\rm dom } (G)$ convex and with its gradient $\nabla G$ being $L_G$-Lipschitz continuous on ${\rm dom } (G)$.
\item[{\rm ($H_2$)}] $H:\RR^{d}\rightarrow(-\infty,+\infty]$ is a nonconvex differentiable function with ${\rm dom } (H)$ convex and with its gradient $\nabla H$ being $L_H$-Lipschitz continuous on ${\rm dom } (H)$.
\item[{\rm ($H_3$)}] $J:\RR^{n}\rightarrow(-\infty,\infty]$ is l.s.c.\ on ${\rm dom }(J)$, and ${\rm dom }(J)$ is a convex set when $\mathbf{U}=\RR^n$; $J$ is local Lipschitz continuous on ${\rm dom }(J) \cap\mathbf{U}$ when $\mathbf{U}$ is closed convex set of $\RR^n$.
\item[{\rm ($H_4$)}] $\Theta:\RR^{n}\rightarrow\RR^m$ is $L^0_{\Theta}$-Lipschitz continuous, i.e., $\|\Theta(u)-\Theta(u')\|\leq L^0_{\Theta}\|u-u'\|,\, \forall u,u'$.
\item[{\rm ($H_5$)}] $\Omega(u)=\left(\Omega_1(u),...,\Omega_m(u)\right)^{\top}$, function $\Omega_j:\RR^n\rightarrow\RR$, has $L_{\Omega_j}$-Lipschitz gradient, $j=1,...,m$, %with constant $L_{\Omega_j}$,
and denote $L_{\Omega}:=\sum\limits_{j=1}^mL_{\Omega_j}$.
\item[{\rm ($H_6$)}] $B\in\RR^{m \times d}$ is a tall %full column rank
matrix, and the image of $\Theta$ is contained in that of $B$, i.e.\ ${\rm Im}(\Theta)\subseteq {\rm Im}(B)$.
\item[{\rm ($H_7$)}] $F$ is lower bounded and coercive over the feasible set
$\{(u,v)\in\RR^n\times\RR^{d}:\Theta(u)+Bv=0\}.$
\end{itemize}
\end{assumption}
By {\rm ($H_1$)} and  {\rm ($H_2$)} in Assumption~\ref{assump1} (cf.~Theorem 3.2.12 of~\cite{Ortega1970}),
we have the following descent inequalities for $G$ and $H$:
\begin{eqnarray}
& & G(u,v)-G(u',v')-\langle\nabla_u G(u',v'),u-u'\rangle-\langle\nabla_v G(u',v'),v-v'\rangle  \nonumber \\
& & \leq  \frac{L_G}{2}\left[\|u-u'\|^2+\|v-v'\|^2\right]; \label{eq:Lip_G} \\
& & G(u,v)-G(u,v')-\langle\nabla_v G(u,v'),v-v'\rangle\leq\frac{L_G}{2}\|v-v'\|^2; \label{eq:Lip_G_v} \\
& & H(v)-H(v')-\langle\nabla H(v'),v-v'\rangle\leq\frac{L_H}{2}\|v-v'\|^2.\label{eq:Lip_H}
\end{eqnarray}
By {\rm ($H_5$)} in Assumption~\ref{assump1} (cf.~Theorem 3.2.12 of~\cite{Ortega1970}), the following inequality holds:
\begin{equation}
\langle p,\Omega(u)-\Omega(u')-\nabla\Omega(u')(u-u')\rangle\leq\frac{\|p\|L_{\Omega}}{2}\|u-u'\|^2 ,\quad \forall p\in\RR^m .\label{eq:Lip_Omega}
\end{equation}

%\section{Nonconvex Auxiliary Problem Principle of Augmented Lagrangian for (P)}

\section{An Auxiliary Problem Principle Approach to (P)}
\label{FBS}

We now propose a new method, to be called {\it nonconvex auxiliary problem principle of multipliers}\/ (NAPP-AL) for finding a stationary point for the non-convex optimization model (P).
%study NC-APP-AL for finding a stationary point of (P).
%Firstly, we
To begin, let us
introduce a Bregman distance function $D(u,u')=K(u)-K(u')-\langle\nabla K(u'),u-u'\rangle$ where $K$ is $\beta$-strongly convex and $L_K$-gradient Lipschitz.
%Moreover, we need the following assumption on core function $K$ and parameter $\gamma$:
Our algorithm is schematically outlined as follows: \\

\noindent\rule[0.25\baselineskip]{\textwidth}{1.5pt}
%{\bf Algorithm: Nonconvex Auxiliary Problem Principle of Augmented Lagrangian (NC-APP-AL) method for (P)}\\
{\bf Nonconvex Auxiliary Problem Principle with Augmented Lagrangian (NAPP-AL)} \\
\noindent\rule[0.25\baselineskip]{\textwidth}{0.5pt}
{Parameters:} $L^0_{\Theta},\, L_{\Omega},\, L_H,\, L_G$, and $\gamma>0, \, \beta,\, 0<\sigma<1$; \\
{Variables:}  $u,\, v,\, p,\, q$. \\
{Initialize:} $u^0\in\RR^{n}$, $v^0\in\RR^{d}$, $p^0\in\RR^m$,  $k:=0$. \\
 \textbf{For} iteration $k$ \textbf{do}
\begin{eqnarray}\label{APk}
& & \mbox{Set $q^k=p^k+\gamma(\Theta(u^k)+Bv^k)$,} \mbox{ and choose $\delta_k>0$ and $\epsilon_k$ satisfying $\sigma \delta_k \le \epsilon^k \le \delta_k$.} \nonumber \\
%& & \delta_k :=\beta \left(L_G+\|q^k\|L_{\Omega}+\gamma (L^0_{\Theta})^2+\frac{14\gamma\|B\|^2(L^0_{\Theta})^2}{\lambda_{\min}(B^{\top}B)}+\frac{14(L_G+\gamma\|B\| L^0_{\Theta})^2}{\gamma\lambda_{\min}(B^{\top}B)}+1\right)^{-1}; \nonumber \\
%& & \frac{\sigma\beta}{L_G+\|q^k\|L_{\Omega}+15\gamma L^0_{\Theta}^2+
%\frac{14}{\gamma}(L_G+\gamma L^0_{\Theta})^2+1} \leq \epsilon^k \leq \frac{\beta}{L_G+\|q^k\|L_{\Omega}+15\gamma L^0_{\Theta}^2+\frac{14}{\gamma}(L_G+\gamma L^0_{\Theta})^2+1}; \\
%\mbox{(AP$^k$)}
& & u^{k+1}\in\arg\min_{u\in\mathbf{U}}\,\, \langle\nabla_u G(u^k,v^k),u\rangle+J(u)+\langle q^k,
\nabla\Omega(u^k) u+\Phi(u)\rangle+\frac{1}{\epsilon^k}D(u,u^k);\label{eq:primal_u}\\
& & v^{k+1}:=\arg\min_{v\in\RR^{d}}\,\, \langle\nabla_v G(u^k,v^k)+\nabla H(v^k),v\rangle+\langle q^k,Bv\rangle+\frac{\gamma}{2}\|B(v-v^k)\|^2;\label{eq:primal_v}\\
& & p^{k+1}:=p^k+\gamma(\Theta(u^{k+1})+Bv^{k+1}); \label{eq:dual} \\
& & k:=k+1. \nonumber
\end{eqnarray}
\textbf{return For}\\
\noindent\rule[0.25\baselineskip]{\textwidth}{1.5pt}

The above is in principle an approximated {\em augmented Lagrangian method}\/ (ALM), where $(u,v)$ is the primal variable and $p$ is the variable for the augmented Lagrangian dual problem. Steps \eqref{eq:primal_u} and \eqref{eq:primal_v} are designed to evaluate the augmented Lagrangian dual function at the given dual variable $p^k$. In \eqref{eq:primal_u}, the $u$ variable is updated by a gradient proximal step where the augmented term $\frac{\gamma}{2}\|\Theta(u)+Bv\|^2$ is also linearized.
Similarly, Step~\eqref{eq:primal_v} is a gradient proximal step on $v$, with the augmented term being linearized.
Vector $q^k$ in \eqref{eq:primal_u} and \eqref{eq:primal_v} incorporates this linearization.

\begin{remark}
%Regarding decomposition, the interesting part of the NAPP-AL algorithm is as follows: for the structured problem (P$_1$)
In the implementation of Step \eqref{eq:primal_u}of NAPP-AL for (P(b)),
 if $J(u)=\sum\limits_{i=1}^NJ_i(u_i)$, $\Phi(u)=\sum\limits_{i=1}^N\Phi_i(u_i)$, $\mathbf{U}=\mathbf{U}_1\times\cdots\times\mathbf{U}_N$, $u_i\in\mathbf{U}_i$, $i=1,...,N$, and if we choose an additive function $K(u)=\sum\limits_{i=1}^NK_i(u_i)$, then the $k$-th iteration %the problem (AP$^k$)
is split into $N$ independent subproblems:
$$
%\mbox{{\rm(AP$_i^k$)}}\qquad
u_i^{k+1}\in\arg\min\limits_{u_i\in\mathbf{U}_i}\langle\left(\nabla_u G(u^k,v^k)\right)_i,u_i\rangle+J_i(u_i)+\langle q^k,\left(\nabla\Omega(u^k)\right)_i  u_i+\Phi_i(u_i)\rangle+\frac{1}{\epsilon^k}D_i(u_i,u_i^k).
$$
Typical choices of $D_i(u_i,u_i^k)$ include $D_i(u_i,u_i^k)=\frac{1}{2}\|u_i-u_i^k\|^2$ or $D_i(u_i,u_i^k)=\frac{1}{2}\|u_i-u_i^k\|_{H_i}^2$.
\end{remark}

\begin{remark}
We shall comment that we assumed Step~\eqref{eq:primal_u} of NAPP-AL, though possibly nonconvex, is solvable to global optimality. Many problems arising from machine learning and statistics are indeed in this category. For example, for nonconvex regularization functions such as the smoothly clipped absolute deviation (SCAD), the minimax concave penalty (MCP), and Capped-$\ell_1$ are separable and admit closed-form solutions for~\eqref{eq:primal_u}.
Furthermore, subproblem~\eqref{eq:primal_v} is a strongly convex quadratic minimization, requiring only one pre-solved matrix factorization.
%has a closed-form solution if $B$ is a diagonal matrix: $v^{k+1}=v^k-\frac{1}{\gamma}\left[\nabla_vG(u^k,v^k)+\nabla H(v^k)+q^k\right]$.
\end{remark}

To facilitate further analysis, let us make the following assumption:
\begin{assumption}\label{assump2}
{\ }

\begin{itemize}
\item [{\rm(i)}] $K$ is strongly convex with parameter $\beta$ and differentiable with its gradient Lipschitz continuous with parameter $L_K$ on $\RR^n$.
\item [{\rm(ii)}] Parameter $\gamma$ is chosen to satisfy: $\gamma>\frac{\sqrt{57}+1}{2\lambda_{\min}(B^{\top}B)}(L_G+L_H).$
\end{itemize}
\end{assumption}

Next proposition reveals a property of $\mathcal{L}_{\gamma}$ in the progression of NAPP-AL.

\begin{proposition} %[Property of $\mathcal{L}_{\gamma}$]
\label{prop}
Suppose that in the implementation of NAPP-AL, the parameters are chosen to satisfy the requirements in Assumptions~\ref{assump1} and~\ref{assump2}, and at each iteration we set
\begin{equation} \label{delta-k}
\delta_k := \beta \left(L_G+\|q^k\|L_{\Omega}+\gamma (L^0_{\Theta})^2+\frac{14\gamma\|B\|^2(L^0_{\Theta})^2}{\lambda_{\min}(B^{\top}B)}+\frac{14(L_G+\gamma\|B\| L^0_{\Theta})^2}{\gamma\lambda_{\min}(B^{\top}B)}+1\right)^{-1}.
\end{equation}
Let the sequence $\{w^{k}\}$ is generated by NAPP-AL.
Then
$${\rm dist}\left(0,\partial\mathcal{L}_{\gamma}(w^{k+1})\right) \leq h(u^k,p^k)\|w^k-w^{k+1}\|,$$
where
%\begin{eqnarray*}
%& &
\[
h(u^k,p^k)
:= \max\left\{2L_G+\|q^k\|L_{\Omega}+\gamma\|B\|L^0_{\Theta}+\gamma (L^0_{\Theta})^2+
\frac{L_K}{\sigma\delta_k},
%[L_G+\|q^k\|L_{\Omega}+15\gamma (L^0_{\Theta})^2+\frac{14}{\gamma}(L_G+\gamma L^0_{\Theta})^2+1]}{\sigma\beta},
%\right. \\
%& & \left.
2L_G+L_H+\gamma\|B\|L^0_{\Theta}, L^0_{\Theta}+\|B\|+\frac{1}{\gamma}\right\}.
\]
\end{proposition}
\begin{proof}
The optimality condition of subproblems~\eqref{eq:primal_u} and~\eqref{eq:primal_v} yields that
\begin{eqnarray*}
0 &\in &\nabla_u G(u^k,v^k)+\partial J(u^{k+1})+\mathcal{N}_{\mathbf{U}}(u^{k+1})+\left(\nabla\Omega(u^{k})+\nabla\Phi(u^{k+1})\right)^{\top}q^k+\frac{1}{\epsilon^k}\left[\nabla K(u^{k+1})-\nabla K(u^k)\right];\\
0 &=& \nabla_v G(u^k,v^k)+\nabla H(v^{k})+B^{\top}q^k+\gamma B^{\top}B(v^{k+1}-v^k).
\end{eqnarray*}

Let us define
\begin{eqnarray*}
\xi_u&=&\nabla_u G(u^{k+1},v^{k+1})-\nabla_u G(u^k,v^k)+\left(\nabla\Omega(u^{k+1})-\nabla\Omega(u^{k})\right)^{\top}q^k+\left(\nabla\Theta(u^{k+1})\right)^{\top}(p^{k+1}-p^k)\nonumber\\
&&+\gamma\left(\nabla\Theta(u^{k+1})\right)^{\top}\left[\left(\Theta(u^{k+1})+Bv^{k+1}\right)-\left(\Theta(u^k)+Bv^k\right)\right]-\frac{1}{\epsilon^k}\left[\nabla K(u^{k+1})-\nabla K(u^k)\right] , \\
\xi_v&=&\nabla_v G(u^{k+1},v^{k+1})-\nabla_v G(u^k,v^k)+\nabla H(v^{k+1})-\nabla H(v^{k})+B^{\top}(p^{k+1}-p^k)\nonumber\\
&&+\gamma B^{\top}\left[\left(\Theta(u^{k+1})+Bv^{k+1}\right)-\left(\Theta(u^k)+Bv^k\right)\right]-\gamma B^{\top}B(v^{k+1}-v^k),
\end{eqnarray*}
and
%\begin{equation}
\[
\xi_p=\frac{1}{\gamma}(p^{k+1}-p^k).
\]
%\end{equation}
Using the above relations, by straightforward verification we conclude that
\begin{eqnarray*}
&&\xi_u\in\partial_u\mathcal{L}_{\gamma}(u^{k+1},v^{k+1},p^{k+1});\\
&&\xi_v=\nabla_v\mathcal{L}_{\gamma}(u^{k+1},v^{k+1},p^{k+1});\\
&&\xi_p=\nabla_p\mathcal{L}_{\gamma}(u^{k+1},v^{k+1},p^{k+1}).
\end{eqnarray*}

By $(H_4)$ in Assumption~\ref{assump1} and the mean value theorem, we obtain that $\|\nabla\Theta(u)\|\leq L^0_{\Theta}$. Therefore, by Assumptions~\ref{assump1} and~\ref{assump2}, we have
\begin{eqnarray*}
\|\xi_u\| & \leq & \left( L_G+\|q^k\|L_{\Omega}+\gamma (L^0_{\Theta})^2+\frac{L_K}{\epsilon^k}\right)\|u^k-u^{k+1}\|+(L_G+\gamma\|B\|L^0_{\Theta})\|v^k-v^{k+1}\|+ L^0_{\Theta}\|p^k-p^{k+1}\|;\\
\|\xi_v\| & \leq & \left( L_G+\gamma\|B\| L^0_{\Theta}\right) \|u^k-u^{k+1}\|+(L_G+L_H)\|v^k-v^{k+1}\|+\|B\|\|p^k-p^{k+1}\|.
\end{eqnarray*}
Observe that $\xi=\left(\begin{array}{c}
\xi_u\\
\xi_v\\
\xi_p
\end{array}\right)$, $q^k=p^k+\gamma\left(\Theta(u^k)+Bv^k\right)$, and $\sigma \delta_k \le \epsilon^k \le \delta_k$,
and so we have
$$
\|\xi\|\leq h(u^k,p^k)\|w^k-w^{k+1}\|,
$$
where
\[
h(u^k,p^k) = \max\left\{2L_G+\|q^k\|L_{\Omega}+\gamma\|B\|L^0_{\Theta}+\gamma (L^0_{\Theta})^2+
\frac{L_K}{\sigma\delta_k},2L_G+L_H+\gamma\|B\|L^0_{\Theta}, L^0_{\Theta}+\|B\|+\frac{1}{\gamma}\right\}.
\]
Combining with $\xi\in\partial\mathcal{L}_{\gamma}(w^{k+1})$, the desired result follows.
\end{proof}

\section{Convergence of NAPP-AL} \label{convergence_FBS}

%Before studying the convergence result, we
In this section, we shall establish a number of lemmas, which lead to a convergence result for NAPP-AL.
\begin{lemma}[bounding the dual variable] \label{lemma:p}
Suppose Assumptions~\ref{assump1} and~\ref{assump2} hold. Let the sequence $\{w^{k}\}$ be generated by NAPP-AL. Then
\begin{itemize}
\item[{\rm(i)}] $\|B^{\top}(p^k-p^{k+1})\|^2\geq\lambda_{\min}(B^{\top}B)\|p^k-p^{k+1}\|^2$
\item[{\rm(ii)}] $B^{\top}p^{k}=-\nabla_v G(u^{k-1},v^{k-1})-\nabla H(v^{k-1})+\gamma B^{\top}[\Theta(u^{k})-\Theta(u^{k-1})]$;
\item[{\rm(iii)}] $\|p^k-p^{k+1}\|^2\leq \frac{3\gamma^2\|B\|^2(L^0_{\Theta})^2}{\lambda_{\min}(B^{\top}B)}\|u^k-u^{k+1}\|^2
    +\frac{3(L_G+\gamma\|B\|L^0_{\Theta})^2}{\lambda_{\min}(B^{\top}B)}\|u^{k-1}-u^{k}\|^2+\frac{3(L_G+L_H)^2}{\lambda_{\min}(B^{\top}B)}\|v^{k-1}-v^{k}\|^2$.
\end{itemize}
\end{lemma}
\begin{proof}
\begin{itemize}
\item[{\rm(i)}] Since ${\rm Im}(\Theta)\subseteq {\rm Im}(B)$ ((H$_6$) of Assumption~\ref{assump1}) and by~\eqref{eq:dual}, we have $\frac{1}{\gamma}(p^{k+1}-p^k)\in {\rm Im}(B)$. The result is straightforward. %By the combination with $B^{\top}B$ is positive definite, the proof of this statement
    (See also Lemma 3 in~\cite{LiuShenGu2019}).
\item[{\rm(ii)}] Since $v^{k+1}$ is optimal to~\eqref{eq:primal_v}, we have
\begin{equation}\label{eq:bound_dp1}
0=\nabla_v G(u^k,v^{k})+\nabla H(v^k)+B^{\top}p^k+\gamma B^{\top}[\Theta(u^k)+Bv^k]+\gamma B^{\top}B(v^{k+1}-v^{k}).
\end{equation}
By the dual updating formula of NAPP-AL~\eqref{eq:dual},
\begin{equation}\label{eq:bound_dp2}
p^{k+1}=p^k+\gamma[\Theta(u^{k+1})+Bv^{k+1}],
\end{equation}
and together~\eqref{eq:bound_dp1} and~\eqref{eq:bound_dp2}, we have
\begin{equation}\label{eq:bound_dp3}
B^{\top}p^{k+1}=-\nabla_v G(u^k,v^{k})-\nabla H(v^k)+\gamma B^{\top}[\Theta(u^{k+1})-\Theta(u^k)].
\end{equation}
\item[{\rm(iii)}] Using Statement (ii) of this lemma and~\eqref{eq:bound_dp3}, we have
\begin{eqnarray*}
\|B^{\top}(p^k-p^{k+1})\|
&\leq&\gamma\|B\| L^0_{\Theta}\|u^k-u^{k+1}\|+(L_G+\gamma\|B\| L^0_{\Theta})\|u^{k-1}-u^{k}\| \\
& & +(L_G+L_H)\|v^{k-1}-v^{k}\|.
\end{eqnarray*}
Combining with Statement (i), the desired inequality follows.
\end{itemize}
\end{proof}

\begin{lemma} %[Variance of $\mathcal{L}_{\gamma}(w^k)$]
\label{lemma_primal} Suppose that Assumptions~\ref{assump1} and~\ref{assump2} hold. % and set $\delta^k$ as in \eqref{delta-k}.
Suppose that the sequence $\{w^{k}\}$ is generated by NAPP-AL. Then the following estimations hold true:
\begin{eqnarray*}
\mbox{{\rm (i)}}&&\mathcal{L}_{\gamma}(u^{k+1},v^{k+1},p^k)-\mathcal{L}_{\gamma}(u^{k},v^k,p^k)\\
&\leq&-\left(\frac{\beta}{2\epsilon^k}-\frac{L_G+\|q^k\|L_{\Omega}+\gamma (L^0_{\Theta})^2}{2}\right)\|u^k-u^{k+1}\|^2-\frac{\gamma\lambda_{\min}(B^{\top}B)-(L_G+L_H)}{2}\|v^k-v^{k+1}\|^2 \\
&&+\frac{1}{2\gamma}\left(\|p^{k-1}-p^k\|^2-\|p^k-p^{k+1}\|^2\right)\nonumber+\frac{1}{\gamma}\|p^k-p^{k+1}\|^2.\\
\mbox{{\rm (ii)}}&&\mathcal{L}_{\gamma}(u^{k+1},v^{k+1},p^{k+1})-\mathcal{L}_{\gamma}(u^{k+1},v^{k+1},p^k) = \frac{1}{\gamma}\|p^k-p^{k+1}\|^2.\\
\mbox{{\rm (iii)}}&&\mathcal{L}_{\gamma}(u^{k+1},v^{k+1},p^{k+1})-\mathcal{L}_{\gamma}(u^{k},v^k,p^k)\nonumber\\
&\leq&-\left(\frac{\beta}{2\epsilon^k}-\frac{L_G+\|q^k\|L_{\Omega}+\gamma (L^0_{\Theta})^2}{2}\right)\|u^k-u^{k+1}\|^2-\frac{\gamma\lambda_{\min}(B^{\top}B)-(L_G+L_H)}{2}\|v^k-v^{k+1}\|^2\\
&&+\frac{1}{2\gamma}\left(\|p^{k-1}-p^k\|^2-\|p^k-p^{k+1}\|^2\right)\nonumber\\
&&+\frac{6\gamma\|B\|^2(L^0_{\Theta})^2}{\lambda_{\min}(B^{\top}B)}\|u^k-u^{k+1}\|^2+\frac{6(L_G+\gamma\|B\|L^0_{\Theta})^2}{\gamma\lambda_{\min}(B^{\top}B)}\|u^{k-1}-u^{k}\|^2+\frac{6(L_G+L_H)^2}{\gamma\lambda_{\min}(B^{\top}B)}\|v^{k-1}-v^{k}\|^2.
\end{eqnarray*}
\end{lemma}

\begin{proof}
%\begin{itemize}
%\item[{\rm(i)}]
{\rm (i).}
Since $u^{k+1}$ and $v^{k+1}$ are the solution of subproblems~\eqref{eq:primal_u} and~\eqref{eq:primal_v} respectively, we have
\begin{eqnarray} \label{eq:primal_bound1}
& & \langle\nabla_u G(u^k,v^k),u^{k+1}-u\rangle+J(u^{k+1})-J(u)+\langle q^k,\nabla\Omega(u^k)(u^{k+1}-u)+\Phi(u^{k+1})-\Phi(u)\rangle\nonumber\\
& & +\frac{1}{\epsilon^k}\left[D(u^{k+1},u^k)-D(u,u^k)\right] \leq 0,\,\,\, \forall u\in\mathbf{U},
\end{eqnarray}
and
\begin{equation}\label{eq:primal_bound2}
\langle\nabla_v G(u^k,v^k)+\nabla H(v^{k}), v^{k+1}-v\rangle+\langle q^k,B(v^{k+1}-v)\rangle+\gamma\langle B(v^{k+1}-v^k),B(v^{k+1}-v)\rangle\leq0, \,\,\, \forall v\in\RR^{d}.
\end{equation}
Take $u=u^k$ and $v=v^k$ in~\eqref{eq:primal_bound1} and~\eqref{eq:primal_bound2} respectively, we have
\begin{eqnarray}\label{eq:primal_bound3}
&&\langle\nabla_u G(u^k,v^k),u^{k+1}-u^k\rangle+J(u^{k+1})-J(u^k)+\langle q^k,\Theta(u^{k+1})-\Theta(u^k)\rangle\nonumber\\
&\leq&-\frac{1}{\epsilon^k}D(u^{k+1},u^k)+\langle q^k,\Omega(u^{k+1})-\Omega(u^k)-\nabla\Omega(u^k)(u^{k+1}-u^k)\rangle\nonumber\\
&\overset{\eqref{eq:Lip_Omega}}{\leq}&-\left(\frac{\beta}{2\epsilon^k}-\frac{\|q^k\|L_{\Omega}}{2}\right) \|u^k-u^{k+1}\|^2
\end{eqnarray}
and
\begin{eqnarray}\label{eq:primal_bound4}
&&\langle\nabla_v G(u^k,v^k), v^{k+1}-v^k\rangle+H(v^{k+1})-H(v^k)+\langle q^k,B(v^{k+1}-v^k)\rangle\nonumber\\
&\leq&-\gamma\|B(v^k-v^{k+1})\|^2+H(v^{k+1})-H(v^k)-\langle\nabla H(v^k),v^{k+1}-v^k\rangle\nonumber\\
& \overset{\eqref{eq:Lip_H}}{\leq} &-\left(\gamma\|B(v^k-v^{k+1})\|^2-\frac{L_H}{2}\|v^k-v^{k+1}\|^2\right). %\qquad\qquad\qquad\qquad\mbox{(by~\eqref{eq:Lip_H})}
\end{eqnarray}
Using~\eqref{eq:primal_bound3} and~\eqref{eq:primal_bound4}, it follows that
\begin{eqnarray}\label{eq:primal_bound5}
&&\mathcal{L}_{\gamma}(u^{k+1},v^{k+1},p^k)-\mathcal{L}_{\gamma}(u^k,v^k,p^k)\nonumber\\
&=&\big{[}G(u^{k+1},v^{k+1})+J(u^{k+1})+H(v^{k+1})+\langle p^k,\Theta(u^{k+1})+Bv^{k+1}\rangle+\frac{\gamma}{2}\|\Theta(u^{k+1})+Bv^{k+1}\|^2\big{]}\nonumber\\
&&-\big{[}G(u^{k},v^{k})+J(u^k)+H(v^k)+\langle p^k,\Theta(u^k)+Bv^k\rangle+\frac{\gamma}{2}\|\Theta(u^k)+Bv^k\|^2\big{]}\nonumber\\
&=&\big{[}J(u^{k+1})-J(u^k)+\langle p^k,\Theta(u^{k+1})-\Theta(u^k)\rangle\big{]}+\big{[}H(v^{k+1})-H(v^k)+\langle p^k,B(v^{k+1}-v^k)\rangle\big{]}\nonumber\\
&&+\big{[}G(u^{k+1},v^{k+1})-G(u^k,v^k)+\frac{\gamma}{2}\|\Theta(u^{k+1})+Bv^{k+1}\|^2-\frac{\gamma}{2}\|\Theta(u^{k})+Bv^{k}\|^2\big{]}\nonumber\\
&=&\bigg{[}\langle\nabla_u G(u^k,v^k),u^{k+1}-u^k\rangle+J(u^{k+1})-J(u^k)+\langle q^k,\Theta(u^{k+1})-\Theta(u^k)\rangle\bigg{]}\nonumber\\
&&+\bigg{[}\langle\nabla_v G(u^k,v^k),v^{k+1}-v^k\rangle+H(v^{k+1})-H(v^k)+\langle q^k,B(v^{k+1}-v^k)\rangle\bigg{]}\nonumber\\
&&+\bigg{[}G(u^{k+1},v^{k+1})-G(u^{k},v^{k})-\langle\nabla_u G(u^k,v^k),u^{k+1}-u^k\rangle-\langle\nabla_v G(u^k,v^k), v^{k+1}-v^k\rangle\bigg{]}\nonumber\\
&&+\bigg{[}\frac{\gamma}{2}\|\Theta(u^{k+1})+Bv^{k+1}\|^2-\frac{\gamma}{2}\|\Theta(u^k)+Bv^{k}\|^2-\langle\gamma[\Theta(u^k)+Bv^k],\Theta(u^{k+1})-\Theta(u^k)+B(v^{k+1}-v^k)\rangle\bigg{]}\nonumber\\
&\leq&-\left(\frac{\beta}{2\epsilon^k}-\frac{L_G+\|q^k\|L_{\Omega}}{2}\right)\|u^k-u^{k+1}\|^2-\left(\gamma\|B(v^k-v^{k+1})\|^2-\frac{L_H+L_G}{2}\|v^k-v^{k+1}\|^2\right)\nonumber\\
&&+\left[\frac{\gamma}{2}\|\Theta(u^{k+1})+Bv^{k+1}\|^2-\frac{\gamma}{2}\|\Theta(u^k)+Bv^{k}\|^2-\langle\gamma [\Theta(u^k)+Bv^k],\Theta(u^{k+1})-\Theta(u^k)+B(v^{k+1}-v^k)\rangle
\right] \nonumber \\
%&&\qquad\qquad\qquad\qquad\qquad\qquad\qquad\qquad\mbox{(by~\eqref{eq:primal_bound3},~\eqref{eq:primal_bound4} and gradient Lipschitz of $G$~\eqref{eq:Lip_G})}
\end{eqnarray}
where the last inequality is due to \eqref{eq:primal_bound3},~\eqref{eq:primal_bound4} and~\eqref{eq:Lip_G}.
Observe that the last term of the right hand side of~\eqref{eq:primal_bound5} can be written as
\begin{eqnarray*}
&&\frac{\gamma}{2}\|\Theta(u^{k+1})+Bv^{k+1}\|^2-\frac{\gamma}{2}\|\Theta(u^k)+Bv^{k}\|^2-\langle\gamma[\Theta(u^k)+Bv^k],\Theta(u^{k+1})-\Theta(u^k)+B(v^{k+1}-v^k)\rangle\nonumber\\
&=&\frac{\gamma}{2}\|\Theta(u^{k+1})+Bv^{k+1}\|^2+\frac{\gamma}{2}\|\Theta(u^k)+Bv^{k}\|^2-\frac{\gamma}{2}\|\Theta(u^{k+1})+Bv^{k}\|^2-\frac{\gamma}{2}\|\Theta(u^{k})+Bv^{k+1}\|^2\nonumber\\
&&+\left[\frac{\gamma}{2}\|\Theta(u^{k+1})+Bv^{k}\|^2-\frac{\gamma}{2}\|\Theta(u^k)+Bv^{k}\|^2-\langle\gamma\left(\Theta(u^k)+Bv^k\right),
\Theta(u^{k+1})-\Theta(u^k)\rangle\right] \nonumber\\
&&+\left[\frac{\gamma}{2}\|\Theta(u^{k})+Bv^{k+1}\|^2-\frac{\gamma}{2}\|\Theta(u^k)+Bv^{k}\|^2-\langle\gamma\left(\Theta(u^k)+Bv^k\right),B(v^{k+1}-v^k)\rangle\right] \nonumber\\
&=&{\gamma}\|\Theta(u^{k+1})+Bv^{k+1}\|^2+\frac{\gamma}{2}\left(\|\Theta(u^k)+Bv^{k}\|^2-\|\Theta(u^{k+1})+Bv^{k+1}\|^2\right)\nonumber\\
&&+\bigg{[}\frac{\gamma}{2}\|\Theta(u^{k+1})+Bv^{k}\|^2-\frac{\gamma}{2}\|\Theta(u^k)+Bv^{k}\|^2-\langle\gamma
\left(\Theta(u^k)+Bv^k\right),\Theta(u^{k+1})-\Theta(u^k)\rangle\bigg{]}\nonumber\\
&&+\left[\frac{\gamma}{2}\|\Theta(u^{k})+Bv^{k+1}\|^2-\frac{\gamma}{2}\|\Theta(u^k)+Bv^{k}\|^2
-\langle\gamma\left(\Theta(u^k)+Bv^k\right),B(v^{k+1}-v^k)\rangle\right] , \nonumber\\
\end{eqnarray*}
which can be further upper bounded by
\begin{eqnarray}\label{eq:primal_bound7}
%%&&\frac{\gamma}{2}\|\Theta(u^{k+1})+v^{k+1}\|^2-\frac{\gamma}{2}\|\Theta(u^k)+v^{k}\|^2-\langle\gamma\left(\Theta(u^k)+v^k\right),\left(\Theta(u^{k+1})-\Theta(u^k)\right)+(v^{k+1}-v^k)\rangle\nonumber\\
%%&=&\frac{\gamma}{2}\|\Theta(u^{k+1})+v^{k+1}\|^2+\frac{\gamma}{2}\|\Theta(u^k)+v^{k}\|^2-\frac{\gamma}{2}\|\Theta(u^{k+1})+v^{k}\|^2-\frac{\gamma}{2}\|\Theta(u^{k})+v^{k+1}\|^2\nonumber\\
%%&&+\bigg{[}\frac{\gamma}{2}\|\Theta(u^{k+1})+v^{k}\|^2-\frac{\gamma}{2}\|\Theta(u^k)+v^{k}\|^2-\langle\gamma\left(\Theta(u^k)+v^k\right),\Theta(u^{k+1})-\Theta(u^k)\rangle\bigg{]}\nonumber\\
%%&&+\bigg{[}\frac{\gamma}{2}\|\Theta(u^{k})+v^{k+1}\|^2-\frac{\gamma}{2}\|\Theta(u^k)+v^{k}\|^2-\langle\gamma\left(\Theta(u^k)+v^k\right),v^{k+1}-v^k\rangle\bigg{]}\nonumber\\
%& &{\gamma}\|\Theta(u^{k+1})+Bv^{k+1}\|^2+\frac{\gamma}{2}\left(\|\Theta(u^k)+Bv^{k}\|^2-\|\Theta(u^{k+1})+Bv^{k+1}\|^2\right)\nonumber\\
%&&+\bigg{[}\frac{\gamma}{2}\|\Theta(u^{k+1})+Bv^{k}\|^2-\frac{\gamma}{2}\|\Theta(u^k)+Bv^{k}\|^2-\langle\gamma
%\left(\Theta(u^k)+Bv^k\right),\Theta(u^{k+1})-\Theta(u^k)\rangle\bigg{]}\nonumber\\
%&&+\left[\frac{\gamma}{2}\|\Theta(u^{k})+Bv^{k+1}\|^2-\frac{\gamma}{2}\|\Theta(u^k)+Bv^{k}\|^2
%-\langle\gamma\left(\Theta(u^k)+Bv^k\right),B(v^{k+1}-v^k)\rangle\right]\nonumber\\
&\leq&{\gamma}\|\Theta(u^{k+1})+Bv^{k+1}\|^2+\frac{\gamma}{2}\left(\|\Theta(u^k)+Bv^{k}\|^2-\|\Theta(u^{k+1})+Bv^{k+1}\|^2\right) \nonumber\\
&&+\frac{\gamma (L^0_{\Theta})^2}{2}\|u^k-u^{k+1}\|^2+\frac{\gamma}{2}\|B(v^k-v^{k+1})\|^2\nonumber\\
&=&\frac{1}{\gamma}\|p^{k}-p^{k+1}\|^2+\frac{1}{2\gamma} \left(\|p^{k-1}-p^k\|^2-\|p^k-p^{k+1}\|^2\right) \nonumber\\
&&+\frac{\gamma (L^0_{\Theta})^2}{2}\|u^k-u^{k+1}\|^2+\frac{\gamma}{2}\|B(v^k-v^{k+1})\|^2
\end{eqnarray}
where the first inequality is due to the gradient Lipschitz for the functions $\frac{\gamma}{2}\|\theta+Bv^k\|^2$ and $\frac{\gamma}{2}\|\Theta(u^k)+Bv\|^2$, and
the second equality is due to the identity $p^{k+1}-p^k=\gamma[\Theta(u^{k+1})+Bv^{k+1}]$.

Combining \eqref{eq:primal_bound5} with~\eqref{eq:primal_bound7} and full column rank of $B$ in (H$_6$) of Assumption~\ref{assump1}, the desired result follows.

%\item[{\rm(ii)}]
{\rm (ii).}
This statement easily follows by observing that
$$
\mathcal{L}_{\gamma}(u^{k+1},v^{k+1},p^{k+1})-\mathcal{L}_{\gamma}(u^{k+1},v^{k+1},p^{k})=\langle p^{k+1}-p^k,\Theta(u^{k+1})+v^{k+1}\rangle=\frac{1}{\gamma}\|p^{k}-p^{k+1}\|^2.
$$
%\item[{\rm(iii)}]

{\rm (iii).}
Summing Statements (i) and (ii), we obtain
\begin{eqnarray*}
&&\mathcal{L}_{\gamma}(u^{k+1},v^{k+1},p^{k+1})-\mathcal{L}_{\gamma}(u^{k},v^k,p^k) \nonumber \\
&\leq&-\left(\frac{\beta}{2\epsilon^k}-\frac{L_G+\|q^k\|L_{\Omega}+\gamma (L^0_{\Theta})^2}{2} \right) \|u^k-u^{k+1}\|^2-\frac{\gamma\lambda_{\min}(B^{\top}B)-(L_G+L_H)}{2}\|v^k-v^{k+1}\|^2\\
&&+\frac{1}{2\gamma}\left(\|p^{k-1}-p^k\|^2-\|p^k-p^{k+1}\|^2\right)+\frac{2}{\gamma}\|p^{k}-p^{k+1}\|^2.
\end{eqnarray*}
By combining with Statement (iii) of Lemma~\ref{lemma:p}, we have the desired result.
%\end{itemize}
\end{proof}
We construct a sequence $\{\Lambda^k\}$ defined as
\begin{equation} \label{Lambda-def}
\Lambda^k := \mathcal{L}_{\gamma}(u^k,v^k,p^k)+c_1 \|u^{k-1}-u^k\|^2+c_2 \|v^{k-1}-v^{k}\|^2+\frac{1}{2\gamma}\|p^{k-1}-p^k\|^2,
\end{equation}
where
\[
c_1:=\frac{7(L_G+\gamma\|B\|L^0_{\Theta})^2}{\gamma\lambda_{\min}(B^{\top}B)}  \mbox{ and } c_2:=\frac{7(L_G+L_H)^2}{\gamma\lambda_{\min}(B^{\top}B)}.
\]
Noted that $\Lambda^k$ is well defined according to the algorithm; it is composed of a combination of the augmented Lagrangian $\mathcal{L}_{\gamma}(w^k)$ and the primal residual ($\|u^{k-1}-u^k\|$ and $\|v^{k-1}-v^k\|$) and the dual residual $\|p^{k-1}-p^k\|$.
%We call $\Lambda^k$ as augmented Lagrangian$^+$, it contains more rich information and satisfies some useful properties. $\Lambda^k$ will play one important role in the convergence analysis.
The sequence $\Lambda^k$ plays the role of a potential value for the iterates produced by the algorithm. In the next few lemmas we shall establish some important properties of this sequence.
%First we estimate the variance of $\Lambda^k$.
\begin{lemma} %[descent inequality of $\Lambda^k$]
\label{lemma_varianceL} Under Assumptions~\ref{assump1} and~\ref{assump2}, and set $\delta^k$ as in \eqref{delta-k}. Suppose that the sequence $\{w^{k}\}$ is generated by the NAPP-AL. Then,
\begin{equation} \label{w-and-Lambda}
\Lambda^{k+1}-\Lambda^{k}\leq -c_3\|w^k-w^{k+1}\|^2,
\end{equation}
with
%\begin{eqnarray*}
\[
c_3 := \min\left\{\frac{1}{2},c_4,\frac{1}{3\gamma}\right\} \mbox{ where }
%c_4 &:=& \frac{\beta}{2\epsilon^k}-\frac{L_G+\|q^k\|L_{\Omega}+\gamma (L^0_{\Theta})^2}{2}-7\gamma (L^0_{\Theta})^2-c_1 \, \left( >\frac{1}{2} \right) \\
c_4 := \frac{\gamma\lambda_{\min}(B^{\top}B)-(L_G+L_H)}{2}-c_2 \, (>0).
\]
%\end{eqnarray*}
%positive number $d_3=\min\{{\mathfrak{a}_1},{\mathfrak{a}_2},\frac{1}{3\gamma}\}$ with ${\mathfrak{a}_1}=\frac{\beta}{2\epsilon^k}-\frac{L_G+\|q^k\|L_{\Omega}+\gamma (L^0_{\Theta})^2}{2}-7\gamma (L^0_{\Theta})^2-d_1>\frac{1}{2}$ and ${\mathfrak{a}_2}=\frac{\gamma-(L_G+L_H)}{2}-d_2>0$.
\end{lemma}

\begin{proof}

By the definition of $\Lambda^k$, we have
\begin{eqnarray}\label{eq:vLambda_bound1}
\Lambda^{k+1}-\Lambda^k&=&\mathcal{L}_{\gamma}(u^{k+1},v^{k+1},p^{k+1})-\mathcal{L}_{\gamma}(u^{k},v^{k},p^{k})\nonumber\\
&&+c_1(\|u^{k}-u^{k+1}\|^2-\|u^{k-1}-u^k\|^2)+c_2(\|v^{k}-v^{k+1}\|^2-\|v^{k-1}-v^k\|^2)\nonumber\\
&&+\frac{1}{2\gamma}\left(\|p^k-p^{k+1}\|^2-\|p^{k-1}-p^{k}\|^2\right).
\end{eqnarray}
Combining Lemma~\ref{lemma_primal} and \eqref{eq:vLambda_bound1} yields
\begin{eqnarray}\label{eq:vLambda_bound2}
& & \Lambda^{k+1}-\Lambda^k \nonumber \\
& \leq &-\left(\frac{\beta}{2\epsilon^k}-\frac{L_G+\|q^k\|L_{\Omega}+\gamma (L^0_{\Theta})^2}{2}\right)
\|u^k-u^{k+1}\|^2-\frac{\gamma\lambda_{\min}(B^{\top}B)-(L_G+L_H)}{2}\|v^k-v^{k+1}\|^2\nonumber\\
&&+\frac{6\gamma\|B\|^2(L^0_{\Theta})^2}{\lambda_{\min}(B^{\top}B)}\|u^{k}-u^{k+1}\|^2+\frac{6(L_G+\gamma\|B\| L^0_{\Theta})^2}{\gamma\lambda_{\min}(B^{\top}B)}\|u^{k-1}-u^k\|^2+\frac{6(L_G+L_H)^2}{\gamma\lambda_{\min}(B^{\top}B)}\|v^{k-1}-v^k\|^2\nonumber\\
&& + c_1(\|u^{k}-u^{k+1}\|^2-\|u^{k-1}-u^k\|^2) + c_2(\|v^{k}-v^{k+1}\|^2-\|v^{k-1}-v^k\|^2) \nonumber \\
%&=&-\left[\frac{\beta}{2\epsilon^k}-\frac{L_G+\|q^k\|L_{\Omega}+\gamma (L^0_{\Theta})^2}{2}-\frac{6\gamma\|B\|^2(L^0_{\Theta})^2}{\lambda_{\min}(B^{\top}B)}-c_1 \right] \|u^k-u^{k+1}\|^2\nonumber\\
%&&-\left[\frac{\gamma\lambda_{\min}(B^{\top}B)-(L_G+L_H)}{2}-c_2 \right] \|v^k-v^{k+1} \|^2 \nonumber \\
%&& - \left[\frac{1}{\gamma\lambda_{\min}(B^{\top}B)}(L_G+\gamma\|B\|L^0_{\Theta})^{2}\|u^{k-1}-u^k\|^2 + \frac{1}{\gamma\lambda_{\min}(B^{\top}B)}(L_G+L_H)^2\|v^{k-1}-v^k\|^2\right] \nonumber \\
&=&-\underbrace{\left[\frac{\beta}{2\epsilon^k}-\frac{L_G+\|q^k\|L_{\Omega}+\gamma (L^0_{\Theta})^2}{2}-\frac{7\gamma\|B\|^2(L^0_{\Theta})^2}{\lambda_{\min}(B^{\top}B)} - c_1\right]}_{\tau_k}\|u^k-u^{k+1}\|^2 \nonumber \\
& & -\underbrace{\left[\frac{\gamma\lambda_{\min}(B^{\top}B)-(L_G+L_H)}{2}-c_2\right]}_{c_4}\|v^k-v^{k+1}\|^2
 \nonumber \\
&&-\frac{1}{3\gamma}\left[\frac{3\gamma^2\|B\|^2(L^0_{\Theta})^2}{\lambda_{\min}(B^{\top}B)}\|u^k-u^{k+1}\|^2+\frac{3(L_G+\gamma\|B\|L^0_{\Theta})^{2}}{\lambda_{\min}(B^{\top}B)}\|u^{k-1}-u^k\|^2+\frac{3(L_G+L_H)^2}{\lambda_{\min}(B^{\top}B)}\|v^{k-1}-v^k\|^2 \right]. \nonumber \\
\end{eqnarray}
Since $0<\epsilon^k\leq\frac{\beta}{L_G+\|q^k\|L_{\Omega}+\gamma (L^0_{\Theta})^2+\frac{14\gamma\|B\|^2(L^0_{\Theta})^2}{\lambda_{\min}(B^{\top}B)}+\frac{14(L_G+\gamma\|B\| L^0_{\Theta})^2}{\gamma\lambda_{\min}(B^{\top}B)}+1}$ and
$c_1=\frac{7(L_G+\gamma\|B\|L^0_{\Theta})^2}{\gamma\lambda_{\min}(B^{\top}B)}$, we have
\[
\tau_k:=\frac{\beta}{2\epsilon^k}-\frac{L_G+\|q^k\|L_{\Omega}+\gamma (L^0_{\Theta})^2}{2}-\frac{7\gamma\|B\|^2(L^0_{\Theta})^2}{\lambda_{\min}(B^{\top}B)} - c_1 \geq \frac{1}{2}.
\]
Now, as $\gamma>\frac{\sqrt{57}+1}{2\lambda_{\min}(B^{\top}B)}(L_G+L_H)$ (Assumption~\ref{assump2}) and $c_2=\frac{7(L_G+L_H)^2}{\gamma\lambda_{\min}(B^{\top}B)}$,
we also have
\begin{eqnarray}
c_4 &:=&\frac{\gamma\lambda_{\min}(B^{\top}B)-(L_G+L_H)}{2}-c_2\nonumber\\
&=&\frac{[\gamma\lambda_{\min}(B^{\top}B)]^2-(L_G+L_H)\gamma\lambda_{\min}(B^{\top}B)-14(L_G+L_H)^2}{2\gamma\lambda_{\min}(B^{\top}B)}>0.
\end{eqnarray}
Therefore, by Statement (ii) of Lemma~\ref{lemma:p}, inequality~\eqref{eq:vLambda_bound2} yields
\begin{eqnarray}
\Lambda^{k+1}-\Lambda^k\leq-\tau_k \|u^k-u^{k+1}\|^2-c_4 \|v^k-v^{k+1}\|^2-\frac{1}{3\gamma}\|p^k-p^{k+1}\|^2,
\end{eqnarray}
and
\begin{eqnarray}
\Lambda^{k+1}-\Lambda^k\leq-c_3\|w^k-w^{k+1}\|^2,
\end{eqnarray}
with $c_3=\min\left\{\frac{1}{2},c_4,\frac{1}{3\gamma}\right\}$, the claim is proven.
\end{proof}

\begin{theorem}[convergence]\label{theo:convergence}
Under Assumptions \ref{assump1} and \ref{assump2}, and set $\delta^k$ as in \eqref{delta-k}. Suppose that the sequence $\{w^k\}$ is generated by NAPP-AL. Then
\begin{itemize}
\item [{\rm(i)}] $\Lambda^k$ is lower bounded and $\lim\limits_{k\rightarrow\infty}\Lambda^k=\Lambda^*$;
\item [{\rm(ii)}] $\{w^k\}$ is bounded and $\|w^k-w^{k+1}\|\rightarrow 0$;
\item [{\rm(iii)}] The sequences $\{\mathcal{L}_{\gamma}(w^k)\}$ and $\{F(u^k,v^k)\}$ converges to $\Lambda^*$, and $\lim\limits_{k\rightarrow\infty}\Theta(u^k)+Bv^k=0$;
\item [{\rm(iv)}] There exists $c_5>0$ with ${\rm dist}\left(0,\partial\mathcal{L}_{\gamma}(u^k)\right) \leq c_5 \|w^k-w^{k+1}\|$
(hence ${\rm dist}\left(0,\partial\mathcal{L}_{\gamma}(u^k)\right) \rightarrow 0$);
\item [{\rm(v)}] Any cluster point $\bar{w}$ of $\{w^k\}$ is a stationary point of (P) and $\mathcal{L}_{\gamma}(\bar{w})=\Lambda^*$.
\end{itemize}
\end{theorem}
\begin{proof}
%\begin{itemize}
%\item [{\rm(i)}]

{\rm (i).}
Since ${\rm Im}(\Theta)\subseteq {\rm Im}(B)$ in Assumption~\ref{assump1}, let $B\tilde{v}^k=-\Theta(u^k)$. Hence $(u^k,\tilde{v}^k)$ is feasible, and
\begin{equation}\label{eq:Lgamma}
\mathcal{L}_{\gamma}(u^k,v^k,p^k)=F(u^k,v^k)+\langle p^k,B(v^k-\tilde{v}^k)\rangle+\frac{\gamma}{2}\|B(v^k-\tilde{v}^k)\|^2.
\end{equation}
Using Lemma~\ref{lemma:p}, we have
\begin{eqnarray}\label{eq:pv}
&&\langle B^{\top}p^k,v^k-\tilde{v}^{k}\rangle\nonumber\\
&=&\langle-\nabla_v G(u^{k-1},v^{k-1})-\nabla H(v^{k-1})+\gamma B^{\top}[\Theta(u^k)-\Theta(u^{k-1})],v^k-\tilde{v}^k\rangle\nonumber\\
&=&\langle\nabla_v G(u^k,v^{k})-\nabla_v G(u^{k-1},v^{k-1})+\nabla H(v^{k})-\nabla H(v^{k-1})+\gamma B^{\top} [\Theta(u^k)-\Theta(u^{k-1})],v^k-\tilde{v}^k\rangle\nonumber\\
&&-\langle\nabla_v G(u^k,v^k)+\nabla H(v^k),v^k-\tilde{v}^k\rangle\nonumber\\
&\geq&-\left[ (L_G+ \gamma\|B\|L^0_{\Theta})\|u^{k-1}-u^{k}\|+(L_G+L_H)\|v^{k-1}-v^k\|\right] \cdot \|v^k-\tilde{v}^k\| \nonumber \\
&&-\langle \nabla_v G(u^k,v^k)+\nabla H(v^k),v^k-\tilde{v}^k\rangle\nonumber\\
&\geq&-\frac{7(L_G+\gamma\|B\| L^0_{\Theta})^2}{\gamma\lambda_{\min}(B^{\top}B)}\|u^{k-1}-u^k\|^2-\frac{7(L_G+L_H)^2}{\gamma\lambda_{\min}(B^{\top}B)}\|v^{k-1}-v^k\|^2-\frac{\gamma\lambda_{\min}(B^{\top}B)}{7}\|v^k-\tilde{v}^k\|^2\nonumber\\
&&-\langle\nabla_v G(u^k,v^k)+\nabla H(v^k),v^k-\tilde{v}^k\rangle\nonumber\\
&\geq&-\frac{7(L_G+\gamma\|B\| L^0_{\Theta})^2}{\gamma\lambda_{\min}(B^{\top}B)}\|u^{k-1}-u^k\|^2-\frac{7(L_G+L_H)^2}{\gamma\lambda_{\min}(B^{\top}B)}\|v^{k-1}-v^k\|^2-\frac{\gamma}{7}\|B(v^k-\tilde{v}^k)\|^2\nonumber\\
&&-\langle\nabla_v G(u^k,v^k)+\nabla H(v^k),v^k-\tilde{v}^k\rangle.
\end{eqnarray}
Together~\eqref{eq:Lgamma} and~\eqref{eq:pv} and $\gamma>\frac{\sqrt{57}+1}{2\lambda_{\min}(B^{\top}B)}(L_G+L_H)$ (Assumption~\ref{assump2}), by the definition of $\Lambda^k$, we obtain that
\begin{eqnarray}\label{eq:Lamdbaklowerbound}
\Lambda^k&\geq&F(u^k,v^k)-\langle\nabla_v G(u^k,v^k)+\nabla H(v^k),v^k-\tilde{v}^k\rangle+\frac{5\gamma}{14}\|B(v^k-\tilde{v}^k)\|^2\nonumber\\
&\geq&F(u^k,v^k)-\langle\nabla_v G(u^k,v^k)+\nabla H(v^k),v^k-\tilde{v}^k\rangle+\frac{5\gamma\lambda_{\min}(B^{\top}B)}{14}\|v^k-\tilde{v}^k\|^2\nonumber\\
&\ge& F(u^k,v^k)-\langle\nabla_v G(u^k,v^k)+\nabla H(v^k),v^k-\tilde{v}^k\rangle+\frac{L_G+L_H}{2}\|v^k-\tilde{v}^k\|^2\nonumber\\
&&+\frac{5\gamma\lambda_{\min}(B^{\top}B)-7(L_G+L_H)}{14\|B\|^2}\|\Theta(u^k)+Bv^k\|^2\nonumber\\
&\overset{\mbox{\eqref{eq:Lip_G_v} and~\eqref{eq:Lip_H}}}{\geq} &
F(u^k,\tilde{v}^k)+\frac{5\gamma\lambda_{\min}(B^{\top}B)-7(L_G+L_H)}{14\|B\|^2}\|\Theta(u^k)+Bv^k\|^2 .  %\qquad\mbox{(by~\eqref{eq:Lip_G_v} and~\eqref{eq:Lip_H})}
%\nonumber\\
%&\geq&F(u^k,\tilde{v}^k).
\end{eqnarray}
By the lower boundedness of $F$  (($H_7$) of Assumption~\ref{assump1}), we have that $F(u^k,\tilde{v}^k)>-\infty$. Then $\Lambda^k$ is lower bounded.
 Moreover, Lemma~\ref{lemma_varianceL} suggests that $\Lambda^k$ is decreasing, hence it has a limit, to be
 %. From the lower bounded of $\Lambda^k$,
denoted by $\Lambda^*$, satisfying  $\lim\limits_{k\rightarrow\infty}\Lambda^k=\Lambda^*$. % and $\Lambda^k$ is bounded.

%\item[{\rm(ii)}]

{\rm (ii).}
%Since~\eqref{eq:Lamdbaklowerbound}, $\Lambda^k$ is bounded and $(u^k,\tilde{v}^k)$ is feasible point, from
By the coercivity of $F$ (($H_7$) of Assumption~\ref{assump1}) and~\eqref{eq:Lamdbaklowerbound}, the sequence $\{(u^k,\tilde{v}^k),v^k\}$ must be bounded.
%Furthermore, by~\eqref{eq:Lamdbaklowerbound}, we have
%\begin{eqnarray}
%\Lambda^k\geq F(u^k,\tilde{v}^k)+\frac{5\gamma-7(L_G+L_H)}{14}\|\Theta(u^k)+v^k\|^2.
%\end{eqnarray}
%The boundedness of $v^k$ can be derived by above inequality, statement (i) and the boundness of $(u^k,\tilde{v}^k)$.
Finally, because $\{ (u^k,v^k)\}$ is bounded, the sequence  $\{p^{k}\}$ is also bounded, due to Statement (i) in Lemma~\ref{lemma:p}.
%By Lemma~\ref{lemma_varianceL} and statement (i) we obtain that
Using \eqref{w-and-Lambda}, we have $\|w^k-w^{k+1}\|\rightarrow0$ as $k \rightarrow \infty$.

%\item [{\rm(iii)}]
{\rm (iii).}
By the definition of $\Lambda^k$ \eqref{Lambda-def} and using Statements (i) and (ii), we have $\mathcal{L}_{\gamma}(w^k)\rightarrow\Lambda^*$.
 Use the definition $\mathcal{L}_{\gamma}(w)$ \eqref{eq:Lgamma} and the boundedness of $p^k$ and fact that $\|p^k-p^{k+1}\|\rightarrow 0$, $\|\Theta(u^k)+Bv^{k}\|\rightarrow 0$, $\mathcal{L}_{\gamma}(w^k)\rightarrow\Lambda^*$, we conclude that $F(u^k,v^k)\rightarrow\Lambda^*$ as $k \rightarrow \infty$.
%\item [{\rm(iv)}]

{\rm (iv).}
By the boundedness of sequence $\{w^k\}$ in Statement (ii) and Proposition~\ref{prop}, we further conclude that there exists
$c_5>0$, such that
$$
{\rm dist}\left(0,\partial\mathcal{L}_{\gamma}(w^k)\right)\leq c_5 \|w^k-w^{k+1}\|.
$$
As a consequence, since $\|w^k-w^{k+1}\|\rightarrow 0$ we have ${\rm dist}\left(0,\partial\mathcal{L}_{\gamma}(w^k)\right) \rightarrow 0$.
%\item [{\rm(v)}]

{\rm (v).}
By Statement (iv) above and $\|w^k-w^{k+1}\|\rightarrow 0$, it is straightforward that any cluster point $\bar{w}$ of $\{w^k\}$ satisfies $0\in\partial\mathcal{L}_{\gamma}(\bar{w})$, i.e., $\bar{w}$ is a stationary point of (P) and $\mathcal{L}_{\gamma}(\bar{w})=\Lambda^*$.
%\end{itemize}
\end{proof}

\begin{proposition} %[Convergence rate analysis]
\label{corollary:rate}
 Under the assumptions of Theorem~\ref{theo:convergence}, it holds that
 %the running best rates\footnote{A nonnegative sequence $a_k$ induces its running best sequence $b_k=\min\{a_i:i\leq k\}$; therefore, $a_k$ has running best rate of $o(1/k)$ if $b_k=o(1/k)$.} of sequence
 \[
 \min_{1\le j \le k} {\rm dist}(0,\partial\mathcal{L}_{\gamma}(w^j))  = o(1/\sqrt{k}).
 \]
\end{proposition}
\begin{proof}
From Lemma~\ref{lemma_varianceL} and Theorem~\ref{theo:convergence}, we have that
$$
c_3\|w^k-w^{k+1}\|^2\leq[\Lambda^k-\Lambda^*]-[\Lambda^{k+1}-\Lambda^*].
$$
Then
$$
\min\limits_{1\leq j\leq k} c_3\|w^j-w^{j+1}\|^2\leq[\Lambda^k-\Lambda^*]-[\Lambda^{k+1}-\Lambda^*].
$$
It follows that
$$
\sum_{k=1}^{+\infty}\min\limits_{1\leq j\leq k} c_3\|w^j-w^{j+1}\|^2\leq\Lambda^1-\Lambda^*<+\infty.
$$
Obviously, $\min\limits_{1\leq j\leq k} c_3\|w^j-w^{j+1}\|^2$ is monotonically non-increasing and $\min\limits_{1\leq j\leq k} c_3\|w^j-w^{j+1}\|^2\geq 0$, by Lemma 1.1 in~\cite{Deng2017}, we have that $\min\limits_{1\leq j\leq k} c_3\|w^j-w^{j+1}\|^2=o(1/k)$. Combining with Statement (iv) in Theorem~\ref{theo:convergence}, the result follows.
\end{proof}

\begin{remark}
Theorem~\ref{theo:convergence} and Proposition~\ref{corollary:rate} imply that to get an $\varepsilon$-stationary point, the number of iterations that the algorithm runs can be upper bounded by:
$k=\frac{c_5^2(\Lambda^1-\Lambda^*)}{c_3\varepsilon^2}=O(1/\varepsilon^2),$
and we can further identify $\hat{k}=\arg\min\limits_{1\leq j\leq k} c_3\|w^j-w^{j+1}\|^2$ such that $w^{\hat{k}}$ is an $\varepsilon$-stationary point of (P).
\end{remark}

%\section{Linear convergence of NAPP-AL with value proximity error bound (VP-EB)}
\section{Linear Convergence Under an Error Bound Condition}
\label{LKL}

In addition to the above iteration complexity of $O(1/\varepsilon^2)$, under some further conditions the convergence of NAPP-AL can actually be {\it linear}. In this section, we will present such an analysis. To this end, in addition to the properties stipulated in Lemma~\ref{lemma_varianceL}, we need to introduce the following value proximity error bound (VP-EB) condition:
\begin{definition} [value proximity error bound (VP-EB) condition] Let $\{w^k\}$ be the sequence generated by NAPP-AL, which converges to $\bar{w}\in\bar{\mathbf{W}}$. We say the value proximity error bound holds at $\bar{w}$, if there exist $\kappa_1>0$, $\eta>0$ and $\nu>0$ such that
\begin{equation*}
\mathcal{L}_{\gamma}(w^{k+1})-\Lambda^*\leq\kappa_1\|w^k-w^{k+1}\|^2, \mbox{ when }
w^{k+1}\in\mathbb{B}(\bar{w};\eta)\cap\{w\in\RR^{n}\times\RR^{d}\times\RR^{m} : \mathcal{L}_{\gamma}(w)<\Lambda^*+\nu\} .
\end{equation*}
\end{definition}
In the above definition, $\bar{\mathbf{W}}$ denotes the set of all stationary points,  $\Lambda^*=\mathcal{L}_{\gamma}(\bar{w})$ and $\mathbb{B}(\bar{w};\eta)$ denotes the open ball of radius $\eta>0$ centered at $\bar{w}$. %Now we are ready to obtain the linear convergence of NC-APP-AL in the following theorem.
Next we shall prove linear convergence of NAPP-AL under this condition.

\begin{theorem} [linear convergence]\label{theo:linearconvergence}
Suppose that the assumptions of Theorem~\ref{theo:convergence} hold, and that $\bar{w}$ is an accumulation point of $\{w^k\}$, and that $\mathcal{L}_{\gamma}(\bar{w})=\Lambda^*$. Furthermore, assume that VP-EB holds at the point $\bar{w}$ with $\eta>0$, $\nu>0$ and $\kappa_1>0$. Then the following statements hold:
\begin{itemize}
\item[{\rm(i)}] There is $k_0$ such that $w^k\in\mathbb{B}(\bar{w};\eta)$ and $\mathcal{L}_{\gamma}(w^k)<\Lambda^*+\nu$, $\forall k\geq k_0$.
\item[{\rm(ii)}] $\sum\limits_{k=0}^{\infty}\|w^k-w^{k+1}\|<+\infty$ (the so-called `finite length property').
\item[{\rm(iii)}] The sequence $\{w^k\}$ actually converges to $\bar{w}$ a stationary point of (P).
\item[{\rm(iv)}] $\{\Lambda^k\}$ converges to $\Lambda^*$ at the Q-linear rate; that is, there are some $\alpha\in(0,1)$ and $k_0$ satisfying
$$\Lambda^{k+1}-\Lambda^*\leq\alpha(\Lambda^k-\Lambda^*),\quad\forall k\geq k_0.$$
\end{itemize}
Moreover, the iterate sequence $\{w^k\}$ itself converges at an R-linear rate to a stationary point $\bar{w}$.
\end{theorem}
\begin{proof}
%\begin{itemize}
%\item[{\rm(i)}]
{\rm (i).}
By Lemma~\ref{lemma_varianceL}, the sequence $\{\Lambda^k\}$ is strictly decreasing, and we have $\Lambda^k>\Lambda^*$, $\forall k$. Using the assumptions of the theorem and the fact that $\mathcal{L}_{\gamma}(w^k)\rightarrow\Lambda^*$, there is a $k_0$ such that
\begin{equation} \label{eq:condition2}
\|w^{k_0}-\bar{w}\|+\frac{2\left(\sqrt{c_3}+\sqrt{\kappa_1+\max\{c_1,c_2,\frac{1}{2\gamma}\}}\right)}{c_3}
\sqrt{\Lambda^{k_0}-\Lambda^*}<\eta,
\end{equation}
and
%\begin{eqnarray}
\[
\mathcal{L}_{\gamma}(w^{k_0})\leq\Lambda^{k_0}<\Lambda^*+\nu.
\]
%\end{eqnarray}
 Now we shall use induction to prove that the sequence $\{w^k\}\subset\mathbb{B}(\bar{w};\eta)$, $\forall k>k_0$. It is clear that $w^{k_0}\in\mathbb{B}(\bar{x};\eta)$ by~\eqref{eq:condition2}. The inequalities $\Lambda^*<\Lambda^{k_0+1}\leq \Lambda^{k_0}<\Lambda^*+\nu$ and $\mathcal{L}_{\gamma}(w^{k_0+1})\leq\Lambda^{k_0+1}<\Lambda^*+\nu$ hold trivially. On the other hand, by Lemma~\ref{lemma_varianceL}, we have
$$
\|w^{k_0}-w^{k_0+1}\|\leq\sqrt{\frac{\Lambda^{k_0}-\Lambda^{k_0+1}}{c_3}}\leq\sqrt{\frac{\Lambda^{k_0}-\Lambda^*}{c_3}}
$$
and
$$
\|w^{k_0+1}-\bar{w}\|\leq\|w^{k_0}-\bar{w}\|+\|w^{k_0}-w^{k_0+1}\|\leq\|w^{k_0}-\bar{w}\|+\sqrt{\frac{\Lambda^{k_0}-\Lambda^*}{c_3}}
\overset{\eqref{eq:condition2}}{<}
\eta.
$$
Thus $w^{k_0+1}\in\mathbb{B}(\bar{w};\eta)$. Now, as hypothesis for the induction, we assume that $w^j\in\mathbb{B}(\bar{w};\eta)$ for $j=k_0+1,..,k_0+k$ and $w^{k_0+k}\neq w^{k_0+k+1}$. Since $\{\Lambda^k\}$ is a strictly decreasing sequence, we have  $\Lambda^*<\Lambda^{k_0+k+1}<\Lambda^{k_0+k}<\cdots<\Lambda^{k_0+2}<\Lambda^{k_0+1}<\Lambda^*+\nu$ and $\mathcal{L}_{\gamma}(w^{j})\leq\Lambda^j<\Lambda^*+\nu$, $\forall j\in\{k_0+1,...,k_0+k+1\}$. To complete the induction we need to show that $w^{k_0+k+1}\in\mathbb{B}(\bar{w};\eta)$.

First, using the concavity of the function $x^{\frac{1}{2}}$ (the gradient inequality), we have
%\begin{eqnarray}
\begin{equation}\label{eq:d1d2}
\left(\Lambda^j-\Lambda^*\right)^{\frac{1}{2}}-\left(\Lambda^{j+1}-\Lambda^*\right)^{\frac{1}{2}}
\geq \frac{\Lambda^j-\Lambda^{j+1}}{2\left(\Lambda^j-\Lambda^*\right)^{\frac{1}{2}}},
\end{equation}
for $j=k_0+1,k_0+2,\dots,k_0+k$.
%\geq \frac{1}{2}\frac{[F(x^i)-F(x^{i+1})]}{\left(F(x^i)-F(\hat{x})\right)^{\frac{p}{2}}}.$$
Since $w^j\in\mathbb{B}(\bar{w};\eta)$ and $\mathcal{L}_{\gamma}(w^j)<\Lambda^*+\nu$, using the VP-EB condition we have
\begin{equation}\label{eq:fvp_KL_j}
\mathcal{L}_{\gamma}(w^{j})-\Lambda^*\leq\kappa_1\|w^{j-1}-w^j\|^2.
\end{equation}
Observe the definition of $\Lambda^k$ \eqref{Lambda-def}, the above inequality leads to
\begin{equation}\label{eq:d1d2_2}
\Lambda^{j}-\Lambda^*\leq \left(\kappa_1+\max\left\{c_1,c_2,\frac{1}{2\gamma}\right\}\right) \|w^{j-1}-w^j\|^2.
\end{equation}
Combining Lemma~\ref{lemma_varianceL},~\eqref{eq:d1d2} and~\eqref{eq:d1d2_2}, we obtain
$$
\frac{2\sqrt{\kappa_1+\max\{c_1,c_2,\frac{1}{2\gamma}\}}}{c_3}\|w^{j-1}-w^j\|
\left[\left(\Lambda^j-\Lambda^*\right)^{\frac{1}{2}}-\left(\Lambda^{j+1}-\Lambda^*\right)^{\frac{1}{2}}\right]
\geq \|w^j-w^{j+1}\|^2.
$$
%It follows from $2\sqrt{\mathfrak{d}_1\mathfrak{d}_2}\leq\mathfrak{d}_1+\mathfrak{d}_2$ with non-negative $\mathfrak{d}_1$ and $\mathfrak{d}_2$ that
Therefore,
\begin{eqnarray}\label{eq:34}
2\|w^j-w^{j+1}\|\leq\|w^{j-1}-w^{j}\|+\frac{2\sqrt{\kappa_1+\max\{c_1,c_2,\frac{1}{2\gamma}\}}}{c_3}
\left[\left(\Lambda^j-\Lambda^*\right)^{\frac{1}{2}}-\left(\Lambda^{j+1})-\Lambda^*\right)^{\frac{1}{2}}\right].
\end{eqnarray}
Summing up~\eqref{eq:34} over $j=k_0+1,...,k_0+k$, we obtain
\begin{eqnarray}\label{eq:51}
&&\sum_{j=k_0+1}^{k_0+k}\|w^j-w^{j+1}\|+\|w^{k_0+k}-w^{k_0+k+1}\|\nonumber\\
&\leq&\|w^{k_0}-w^{k_0+1}\|+\frac{2\sqrt{\kappa_1+\max\{c_1,c_2,\frac{1}{2\gamma}\}}}{c_3}
\left[\left(\Lambda^{k_0+1}-\Lambda^*\right)^{\frac{1}{2}}-\left(\Lambda^{k_0+k+1}-\Lambda^*\right)^{\frac{1}{2}}\right].\nonumber\\
\end{eqnarray}
Using~\eqref{eq:51} along with the triangle inequality, we have
\begin{eqnarray*}
\|w^{k_0+k+1}-\bar{w}\|&\leq&\|w^{k_0}-\bar{w}\|+\|w^{k_0}-w^{k_0+1}\|+\sum_{j=k_0+1}^{k_0+k}\|w^j-w^{j+1}\|\\      &\leq&\|w^{k_0}-\bar{w}\|+2\|w^{k_0}-w^{k_0+1}\|+\frac{2\sqrt{\kappa_1+\max\{c_1,c_2,\frac{1}{2\gamma}\}}}{c_3}
\left(\Lambda^{k_0+1}-\Lambda^*\right)^{\frac{1}{2}}\\    &\leq&\|w^{k_0}-\bar{w}\|+\frac{2\sqrt{c_3}+2\sqrt{\kappa_1+\max\{c_1,c_2,\frac{1}{2\gamma}\}}}{c_3}\left(\Lambda^{k_0}-\Lambda^*\right)^{\frac{1}{2}}\\
&\overset{\eqref{eq:condition2}}{<} &\eta . %~~~(\mbox{by}~\eqref{eq:condition2}).
\end{eqnarray*}
This shows that $w^{k_0+k+1}\in\mathfrak{B}(\bar{w};\eta)$, and
(i) is thus proven by induction. %the Principle of Mathematical Induction.
%\item[{\rm(ii)}] and \rm{(iii)}:

{\rm (ii)-(iii).}
A direct consequence of \eqref{eq:51} is, for all $k$,
$$\sum_{j=k_0+1}^{k_0+k}\|w^j-w^{j+1}\|\leq\|w^{k_0}-w^{k_0+1}\|+\frac{2\sqrt{\kappa_1+\max\{c_1,c_2,\frac{1}{2\gamma}\}}}{c_3}
\left(\Lambda^{k_0+1}-\Lambda^*\right)^{\frac{1}{2}}<+\infty.$$
Therefore
$$\sum_{k=0}^{+\infty}\|w^k-w^{k+1}\|<+\infty.$$
In particular, this implies that the whole sequence $\{w^k\}$ actually converges to the point $\bar{w}$, and that $\bar{w}$ is a stationary point of (P) by Theorem~\ref{theo:convergence}.

%\item[{\rm(iv)}]
{\rm (iv).}
By the combination of Lemma~\ref{lemma_varianceL} and~\eqref{eq:d1d2_2}, we have that $\forall k\geq k_0$,
\begin{eqnarray}
\Lambda^{k+1}-\Lambda^*&=&\Lambda^{k}-\Lambda^*+(\Lambda^{k+1}-\Lambda^{k})\nonumber\\
&\overset{\eqref{w-and-Lambda}}{\leq}&\Lambda^{k}-\Lambda^*-c_3\|w^k-w^{k+1}\|^2\nonumber\\
&\overset{\eqref{eq:d1d2_2}}{\leq}&\Lambda^{k}-\Lambda^*-\frac{c_3}{\kappa_1+\max\{c_1,c_2,\frac{1}{2\gamma}\}}(\Lambda^{k+1}-\Lambda^*).
\end{eqnarray}
Therefore
\begin{eqnarray}
\Lambda^{k+1}-\Lambda^*\leq\alpha(\Lambda^{k}-\Lambda^*),\quad\forall k\geq k_0
\end{eqnarray}
where $\alpha:=\frac{1}{1+\frac{c_3}{\kappa_1+\max\{c_1,c_2,\frac{1}{2\gamma}\}}}\in(0,1)$.

It follows that $\Lambda^k-\Lambda^*\leq\alpha^{k-k_0}(\Lambda^{k_0}-\Lambda^*)$, $\forall k\geq k_0$. By Lemma~\ref{lemma_varianceL}, we have
\begin{eqnarray}
\|w^k-w^{k+1}\|^2&\leq&\frac{1}{c_3}[(\Lambda^{k}-\Lambda^*)-(\Lambda^{k+1}-\Lambda^*)]\nonumber\\
&\leq&\frac{1}{c_3}(\Lambda^{k}-\Lambda^*)\nonumber\\
&\leq&\frac{\alpha^{k-k_0}}{c_3}(\Lambda^{k_0}-\Lambda^*).
\end{eqnarray}
Therefore, $\|w^k-w^{k+1}\|\leq\hat{M}(\sqrt{\alpha})^{k-k_0}$, $\forall k>k_0$, with $\hat{M}=\sqrt{\frac{\Lambda^{k_0}-\Lambda^*}{c_3}}$. By Statement (iii), we conclude that $\{w^k\}$ converges to a desired stationary point $\bar{w}$. Moreover,
$$
\|w^k-\bar{w}\|\leq\sum\limits_{j=k}^{+\infty}\|w^{j}-w^{j+1}\|\leq\frac{\hat{M}}{(1-\sqrt{\alpha})(\alpha)^{k_0/2}}\,(\sqrt{\alpha})^{k},
$$
showing that $\{w^k\}$ converges to $\bar{w}$ at an R-linear rate.
%; that is,
%$$
%\lim_{k\rightarrow+\infty}\sup\sqrt[k-k_0]{\|w^k-\bar{w}\|}=\sqrt{\alpha}<1.
%$$
%\end{itemize}
\end{proof}

\section{On the VP-EB Condition}\label{LMS}

As we have observed, the error bound condition VP-EB leads to linear convergence of NAPP-AL. A natural question arises: {\it Can VP-EB ever be satisfied in a natural setting?} In this section we shall show that at least under two other popular conditions, VP-EB is indeed satisfied.

For given positive numbers $\eta$ and $\nu$, let us define
$$\mathfrak{B}(\bar{w};\eta,\nu)=\mathbb{B}(\bar{w};\eta)\cap\{w\in\RR^{n}\times\RR^{d}\times\RR^m : \Lambda^*<\mathcal{L}_{\gamma}(w)<\Lambda^*+\nu\}.$$
\begin{definition}[Kurdyka-{\L}ojasiewicz property~\cite{Attouch2013,LiPong2018}] The proper lower semicontinuous function $\mathcal{F}$ is said to satisfy the Kurdyka-{\L}ojasiewicz (K{\L}) property at $\bar{x}$ with exponent $\theta\in(0,1)$, if there exist $\eta>0$, $\nu>0$, and $\kappa_2>0$ such that the following inequality holds:
$$[\mathcal{F}(x)-\mathcal{F}(\bar{x})]^{\theta}\leq\kappa_2 \, {\rm dist}(0,\partial\mathcal{F}(x)),\quad\forall x\in\mathfrak{B}(\bar{x};\eta,\nu).$$
\end{definition}
\begin{proposition}[K{\L} property implies VP-EB]
Let the sequence $\{w^k\}$ be generated by NAPP-AL and $\bar{w}$ be an accumulation point of $\{w^k\}$. If $\mathcal{L}_{\gamma}$ satisfies K{\L} property at point $\bar{w}$ with exponent $\theta=\frac{1}{2}$, $\eta>0$, $\nu>0$ and $\kappa_2>0$, then the VP-EB property holds at $\bar{w}$.
\end{proposition}
\begin{proof}
For given $w^{k+1}\in\mathbb{B}(\bar{w};\eta)$ and $\mathcal{L}_{\gamma}(w^{k+1})<\Lambda^*+\nu$, we have two cases to consider here.

%\begin{itemize}
{\bf Case 1.}
If $w^{k+1}\in\mathbb{B}(\bar{w};\eta)\cap\{w\in\RR^{n\times d\times m} : \Lambda^*<\mathcal{L}_{\gamma}(w)<\Lambda^*+\nu\}$ and $\mathcal{L}_{\gamma}(w)$ satisfies K{\L} property at the point $\bar{w}$ with exponent $\theta=\frac{1}{2}$, $\eta>0$, $\nu>0$ and $\kappa_2>0$, by Statement (iv) of Theorem~\ref{theo:convergence}, we obtain that
%    \begin{equation}
\[
\mathcal{L}_{\gamma}(w^{k+1})-\Lambda^*\leq(\kappa_2)^2 {\rm dist}^2\left(0,\mathcal{L}_{\gamma}(w^{k+1})\right) \leq(\kappa_2 c_5)^2\|w^k-w^{k+1}\|^2.
\]
%\end{equation}

{\bf Case 2.} For the case $w^{k+1}\in\mathbb{B}(\bar{w};\eta)\cap\{w\in\RR^{n\times d\times m} : \mathcal{L}_{\gamma}(w)<\Lambda^*\}$, we trivially have
%\begin{equation}
\[
\mathcal{L}_{\gamma}(w^{k+1})-\Lambda^* \leq (\kappa_2 c_5)^2\|w^k-w^{k+1}\|^2.
\]
%\end{equation}
%\end{itemize}
Therefore, for both cases, we have
%\begin{equation}
\[
\mathcal{L}_{\gamma}(w^{k+1})-\Lambda^*\leq(\kappa_2 c_5)^2\|w^k-w^{k+1}\|^2.
\]
%\end{equation}
\end{proof}

Next, we introduce the following metric-subregularity condition.

\begin{definition}[metric subregularity] The set-valued mapping $\mathcal{H}(w)$ is called metric subregularity around $(\bar{w},0)$ if there is a neighborhood $\mathbb{B}(\bar{w};\eta)$ of $\bar{w}$ and $\kappa_3>0$ such that
$$
{\rm dist}(w,\mathcal{H}^{-1}(0))\leq\kappa_3 \, {\rm dist}(0,\mathcal{H}(w)),\qquad\forall w\in\mathbb{B}(\bar{w};\eta).
$$
\end{definition}

To related the metric subregularity with VP-EB, we make the following assumption:
\begin{assumption}\label{assump3}
For $\bar{w}\in\bar{\mathbf{W}}$, there is $\delta>0$ such that $\mathcal{L}_{\gamma}(w)\leq\mathcal{L}_{\gamma}(\bar{w})$ whenever $w\in\bar{\mathbf{W}}$ and $\|w-\bar{w}\|\leq\delta$.
\end{assumption}
Note that if $\bar w$ is an isolated saddle point, then Assumption~\ref{assump3} holds true trivially.

\begin{lemma}[cost-to-go inequality~\cite{TsengYun2009}]\label{lemma:cost-to-go} Let $\{w^k\}$ be a sequence generated by NAPP-AL and $\bar{w}$ be one stationary point of (P). Then there exists $c_6>0$, such that
\begin{equation}\label{eq:ctg}
\mathcal{L}_{\gamma}(w^{k+1})-\mathcal{L}_{\gamma}(\bar{w})\leq c_6 \left(\|\bar{w}-w^{k+1}\|^2+\|w^k-w^{k+1}\|^2\right).
\end{equation}
\end{lemma}
\begin{proof} By the fact that $\Theta(\bar{u})+B\bar{v}=0$, and
\begin{eqnarray*}
0 &\ge& \langle\nabla_uG(u^k,v^k),u^{k+1}-\bar{u}\rangle+J(u^{k+1})-J(\bar{u}) \\
& & +\langle q^k,\nabla\Omega(u^k)(u^{k+1}-\bar{u})+\Phi(u^{k+1})-\Phi(\bar{u})\rangle+\frac{1}{\epsilon^k}\left[ D(u^{k+1},u^k)-D(\bar{u},u^k)\right],
\end{eqnarray*}
and
\[
\langle\nabla_v G(u^k,v^k)+\nabla H(v^k),v^{k+1}-\bar{v}\rangle+\langle q^k,B(v ^{k+1}-\bar{v})\rangle+\frac{\gamma}{2}\left[\|B(v^k-v^{k+1})\|^2-\|B(\bar{v}-v^k)\|^2 \right]
\leq0,
\]
we obtain
\begin{eqnarray*}
&&\mathcal{L}_{\gamma}(w^{k+1})-\mathcal{L}_{\gamma}(\bar{w})\nonumber\\
&=&\langle p^{k+1}, \Theta(u^{k+1})-\Theta(\bar{u})+B(v^{k+1}-\bar{v})\rangle+F(u^{k+1},v^{k+1})-F(\bar{u},\bar{v})+\frac{\gamma}{2}\|\Theta(u^{k+1})+Bv^{k+1}\|^2\nonumber\\
&=&\langle p^{k+1}-q^k, \left(\Theta(u^{k+1})-\Theta(\bar{u})\right)+B(v^{k+1}-\bar{v})\rangle\nonumber\\
&&+\left\{\langle\nabla_uG(u^k,v^k),u^{k+1}-\bar{u}\rangle+J(u^{k+1})-J(\bar{u})+\langle q^k,\nabla\Omega(u^k)(u^{k+1}-\bar{u})+\Phi(u^{k+1})-\Phi(\bar{u})\rangle \right. \nonumber\\
&&+\left. \frac{1}{\epsilon^k}[D(u^{k+1},u^k)-D(\bar{u},u^k)]\right\}+\langle q^k,\Omega(u^{k+1})-\Omega(\bar{u})-\nabla\Omega(u^k)(u^{k+1}-\bar{u})\rangle\nonumber\\
&&+\left\{\langle\nabla_vG(u^k,v^k)+\nabla H(v^k),v^{k+1}-\bar{v}\rangle+\langle q^k,B(v^{k+1}-\bar{v})\rangle
+\frac{\gamma}{2}\left[\|B(v^k-v^{k+1})\|^2-\|B(\bar{v}-v^k)\|^2
\right]\right\} \nonumber\\
&&+\left\{ G(u^{k+1},v^{k+1})-G(\bar{u},\bar{v})-\langle\nabla_uG(u^k,v^k),u^{k+1}-\bar{u}\rangle-\langle\nabla_vG(u^k,v^k),v^{k+1}-\bar{v}\rangle\right\}\nonumber\\
&&+\left\{ H(v^{k+1})-H(\bar{v})-\langle\nabla H(v^k),v^{k+1}-\bar{v}\rangle\right\}\nonumber\\
&&-\frac{1}{\epsilon^k}[D(u^{k+1},u^k)-D(\bar{u},u^k)]-\frac{\gamma}{2}[\|B(v^k-v^{k+1})\|^2-\|B(\bar{v}-v^k)\|^2]+\frac{\gamma}{2}\|\Theta(u^{k+1})+Bv^{k+1}\|^2\nonumber
\end{eqnarray*}
which further leads to an upper bound
\begin{eqnarray}\label{eq:Lbw}
&&\mathcal{L}_{\gamma}(w^{k+1})-\mathcal{L}_{\gamma}(\bar{w})\nonumber\\
&\leq&\underbrace{\langle p^{k+1}-q^k,\left(\Theta(u^{k+1})-\Theta(\bar{u})\right)+B(v^{k+1}-\bar{v})\rangle}_{{\cal T}_1} %\nonumber\\
+\underbrace{\langle q^k,\Omega(u^{k+1})-\Omega(\bar{u})-\nabla\Omega(u^k)(u^{k+1}-\bar{u})\rangle}_{{\cal T}_2}\nonumber\\
&&+\underbrace{\left\{G(u^{k+1},v^{k+1})-G(\bar{u},\bar{v})-\langle\nabla_uG(u^k,v^k),u^{k+1}-\bar{u}\rangle
-\langle\nabla_vG(u^k,v^k),v^{k+1}-\bar{v}\rangle\right\}}_{{\cal T}_3}\nonumber\\
&&+\underbrace{\left\{H(v^{k+1})-H(\bar{v})-\langle\nabla H(v^k),v^{k+1}-\bar{v}\rangle\right\}}_{{\cal T}_4}\nonumber\\
&&+\underbrace{\frac{1}{\epsilon^k}D(\bar{u},u^k)+\frac{\gamma}{2}\|B(\bar{v}-v^k)\|^2+\frac{\gamma}{2}\|\Theta(u^{k+1})+Bv^{k+1}\|^2}_{{\cal T}_5}.
\end{eqnarray}
Using the facts that $q^k=p^k+\gamma\left(\Theta(u^k)+Bv^k\right)$ and $p^{k+1}-p^k=\gamma\left(\Theta(u^{k+1})+Bv^{k+1}\right)$,
we may further upper bound the term ${\cal T}_1$:
\begin{eqnarray}\label{eq:b1}
{\cal T}_1 &=&\langle p^{k+1}-q^k, \Theta(u^{k+1})-\Theta(\bar{u})+B(v^{k+1}-\bar{v})\rangle\nonumber\\
&=&\langle\gamma[\Theta(u^{k+1})-\Theta(u^k)+B(v^{k+1}-v^k)],\left(\Theta(u^{k+1})-\Theta(\bar{u})\right)+B(v^{k+1}-\bar{v})\rangle\nonumber\\
&\leq&\gamma^2(L^0_{\Theta})^2\|u^k-u^{k+1}\|^2+\gamma^2\|B\|^2\|v^k-v^{k+1}\|^2+ (L^0_{\Theta})^2\|\bar{u}-u^{k+1}\|^2+\|B\|^2\|\bar{v}-v^{k+1}\|^2.\nonumber\\
\end{eqnarray}
Due to the boundedness of sequence $\{w^k\}$ as stipulated by Statement (ii) of Theorem~\ref{theo:convergence}, there exist positive constants $c_7$ and $c_8$ satisfying:
\[
c_7 >\|q^k\|L_{\Omega}/2 \mbox{ and } c_8 >5 \|q^k\|L_{\Omega}/2 .
\]
Therefore, by {\rm(H$_5$)} of Assumption~\ref{assump1}, we may upper bound the term ${\cal T}_2$ as
\begin{eqnarray}\label{eq:b1_1}
{\cal T}_2 &=&\langle q^k,\Omega(u^{k+1})-\Omega(\bar{u})-\nabla\Omega(u^k)(u^{k+1}-\bar{u})\rangle\nonumber\\
&=&\langle q^k,\Omega(u^{k+1})-\Omega(\bar{u})-\nabla\Omega(\bar{u})(u^{k+1}-\bar{u})\rangle+\langle q^k,[\nabla\Omega(\bar{u})-\nabla\Omega(u^k)](u^{k+1}-\bar{u})\rangle\nonumber\\
&\leq&\|q^k\|L_{\Omega}\|\bar{u}-u^{k+1}\|^2+\|q^k\|L_{\Omega}\|\bar{u}-u^k\|\|\bar{u}-u^{k+1}\|\nonumber\\
&\leq&\|q^k\|L_{\Omega}\|\bar{u}-u^{k+1}\|^2+\|q^k\|L_{\Omega}(\|\bar{u}-u^{k+1}\|+\|u^{k}-u^{k+1}\|)\|\bar{u}-u^{k+1}\|\nonumber\\
&\leq&\|q^k\|L_{\Omega}\left[ 2\|\bar{u}-u^{k+1}\|^2+ \frac{1}{2}\left(\|u^{k}-u^{k+1}\|^2 + \|\bar{u}-u^{k+1}\|^2\right) \right]\nonumber\\
&\leq& c_7 \|u^k-u^{k+1}\|^2+c_8 \|\bar{u}-u^{k+1}\|^2.
\end{eqnarray}
Next, we use the gradient Lipschitz property of $G$ in Assumption~\ref{assump1} to obtain
\begin{eqnarray}\label{eq:b2}
{\cal T}_3 &=&G(u^{k+1},v^{k+1})-G(\bar{u},\bar{v})-\langle\nabla_uG(u^k,v^k),u^{k+1}-\bar{u}\rangle-\langle\nabla_vG(u^k,v^k),v^{k+1}-\bar{v}\rangle\nonumber\\
&=&G(u^{k+1},v^{k+1})-G(\bar{u},\bar{v})-\langle\nabla_uG(\bar{u},\bar{v}),u^{k+1}-\bar{u}\rangle-\langle\nabla_vG(\bar{u},\bar{v}),v^{k+1}-\bar{v}\rangle\nonumber\\
&&+\langle\nabla_uG(\bar{u},\bar{v})-\nabla_uG(u^{k+1},v^{k+1}),u^{k+1}-\bar{u}\rangle+\langle\nabla_vG(\bar{u},\bar{v})-\nabla_vG(u^{k+1},v^{k+1}),v^{k+1}-\bar{v}\rangle\nonumber\\
&&+\langle\nabla_uG(u^{k+1},v^{k+1})-\nabla_uG(u^{k},v^{k}),u^{k+1}-\bar{u}\rangle+\langle\nabla_vG(u^{k+1},v^{k+1})-\nabla_vG(u^{k},v^{k}),v^{k+1}-\bar{v}\rangle\nonumber\\
&\leq& 2L_G\left[\|\bar{u}-u^{k+1}\|^2+\|\bar{v}-v^{k+1}\|^2\right]+\frac{L_G}{2}\left[\|u^{k}-u^{k+1}\|^2+\|v^k-v^{k+1}\|^2\right].
\end{eqnarray}
By the gradient Lipschitz of $H$ in Assumption~\ref{assump1}, we have
\begin{eqnarray}\label{eq:b3}
{\cal T}_4 &=&H(v^{k+1})-H(\bar{v})-\langle\nabla H(v^k),v^{k+1}-\bar{v}\rangle\nonumber\\
&=&H(v^{k+1})-H(\bar{v})-\langle\nabla H(\bar{v}),v^{k+1}-\bar{v}\rangle+\langle\nabla H(\bar{v})-\nabla H(v^{k+1}),v^{k+1}-\bar{v}\rangle\nonumber\\
&&+\langle\nabla H(v^{k+1})-\nabla H(v^k),v^{k+1}-\bar{v}\rangle\nonumber\\
&\le & \frac{L_H}{2} \|\bar{v}-v^{k+1}\|^2 + L_H \|\bar{v}-v^{k+1}\|^2 + L_H \|v^k-v^{k+1}\| \|\bar{v}-v^{k+1}\| \nonumber \\
&\leq & 2L_H \|\bar{v}-v^{k+1}\|^2+\frac{L_H}{2}\|v^k-v^{k+1}\|^2.
\end{eqnarray}
Finally, by the boundedness of the sequence $\{w^k\}$ according to Statement (ii) of Theorem~\ref{theo:convergence},  and $\sigma \delta_k \le \epsilon^k \le \delta_k$ with
$\delta_k =\beta \left( L_G+\|q^k\|L_{\Omega}+\gamma (L^0_{\Theta})^2+\frac{14\gamma\|B\|^2(L^0_{\Theta})^2}{\lambda_{\min}(B^{\top}B)}+\frac{14(L_G+\gamma\|B\| L^0_{\Theta})^2}{\gamma\lambda_{\min}(B^{\top}B)}+1\right)^{-1}$,
there exists $\underline{\epsilon}>0$ such that $\epsilon^k\ge \underline{\epsilon}$ for all $k$.
%$$
%\frac{\sigma\beta}{L_G+\|q^k\|L_{\Omega}+15\gamma (L^0_{\Theta})^2+\frac{14}{\gamma}(L_G+\gamma L^0_{\Theta})^2+1}\leq\epsilon^k\leq\frac{\beta}{L_G+\|q^k\|L_{\Omega}+15\gamma (L^0_{\Theta})^2+\frac{14}{\gamma}(L_G+\gamma L^0_{\Theta})^2+1}
%$$
%we have $0<\underline{\epsilon}\leq\epsilon^k$, $\forall k\in\mathbb{N}$.
Since $p^{k+1}-p^k=\gamma\left(\Theta(u^{k+1})+Bv^{k+1}\right)$, term ${\cal T}_5$ can now be upper bounded as follows:
\begin{eqnarray}\label{eq:b4}
{\cal T}_5 &=&\frac{1}{\epsilon^k}D(\bar{u},u^k)+\frac{\gamma}{2}\|B(\bar{v}-v^k)\|^2+\frac{\gamma}{2}\|\Theta(u^{k+1})+Bv^{k+1}\|^2 \nonumber \\
& \leq &\frac{L_K}{2\underline{\epsilon}}\|\bar{u}-u^k\|^2+\frac{\gamma\|B\|^2}{2}\|\bar{v}-v^k\|^2+\frac{1}{2\gamma}\|p^{k}-p^{k+1}\|^2 \nonumber \\
& \leq &\frac{L_K}{\underline{\epsilon}}\left[\|\bar{u}-u^{k+1}\|^2+\|u^k-u^{k+1}\|^2\right]+\gamma\|B\|^2\left[\|\bar{v}-v^{k+1}\|^2+\|v^k-v^{k+1}\|^2\right] \nonumber \\
&      & +\frac{1}{2\gamma}\|p^{k}-p^{k+1}\|^2 .
\end{eqnarray}
Now, substituting the upper bounds~\eqref{eq:b1},~\eqref{eq:b1_1},~\eqref{eq:b2},~\eqref{eq:b3} and~\eqref{eq:b4} in~\eqref{eq:Lbw}, we readily derive $c_6$ to satisfy \eqref{eq:ctg}.
\end{proof}

\begin{proposition}[metric subregularity implies VP-EB]\label{lemma:ms-linear}
Let $\{w^k\}$ be generated by NAPP-AL, $\bar{w}$ be one cluster point of $\{w^k\}$. If Assumption~\ref{assump3} holds with $\delta$, the mapping $\partial\mathcal{L}_{\gamma}(w)$ is metric subregularity around $(\bar{w},0)$ with $\kappa_3>0$ and $\eta>0$ ($\eta<\delta$), then the VP-EB holds at $\bar{w}$.
\end{proposition}
\begin{proof}
Let $w_p^{k+1}\in\bar{\mathbf{W}}$ such that $\|w^{k+1}-w_p^{k+1}\|={\rm dist}(w^{k+1},\bar{\mathbf{W}})$. By the cost-to-go inequality in Lemma~\ref{lemma:cost-to-go} with $\bar{w}=w_p^{k+1}$, we have
\begin{equation}\label{eq:ms1}
\mathcal{L}_{\gamma}(w^{k+1})-\mathcal{L}_{\gamma}(w_p^{k+1})\leq c_6 \left(\|w^{k+1}-w_p^{k+1}\|^2+\|w^k-w^{k+1}\|^2\right).
\end{equation}
Since $w^{k+1}\in\mathbb{B}(\bar{w};\eta)$, we have that $w_p^{k+1}\in\mathbb{B}(\bar{w};\eta)$. Because $\partial\mathcal{L}_{\gamma}(w)$ is metric subregular around $(\bar{w},0)$, there exists $\mathbb{B}(\bar{w};\eta)$ such that if $w^{k+1}\in\mathbb{B}(\bar{w};\eta)$ then \begin{eqnarray}\label{eq:ms2}
\|w^{k+1}-w_p^{k+1}\|^2={\rm dist}^2(w^{k+1},\bar{\mathbf{W}})&\leq&(\kappa_3)^2 {\rm dist}^2(0,\partial\mathcal{L}_{\gamma}(w^{k+1})) \nonumber \\
%\qquad\mbox{if}\quad w^{k+1}\in\mathbb{B}(\bar{w};\eta)\nonumber\\
&\overset{\mbox{(iv) of Theorem~\ref{theo:convergence}}}{\leq} &(\kappa_3 c_5)^2\|w^k-w^{k+1}\|^2 . %\quad\mbox{(by (iv) of Theorem~\ref{theo:convergence})}
\end{eqnarray}

Now, under the condition that Assumption 3 holds with $\delta\geq\eta$, and by the definition of $w_p^{k+1}$ and the fact $w_p^{k+1}\in\mathbb{B}(\bar{w};\eta)$, %from Assumption 3,
 we conclude that $\mathcal{L}_{\gamma}(w_p^{k+1}) \le \mathcal{L}_{\gamma}(\bar{w})=\Lambda^*$. Using~\eqref{eq:ms1} and~\eqref{eq:ms2}, %there is $c_9>0$ such that the following inequality holds:
we have
$$
\mathcal{L}_{\gamma}(w^{k+1})-\Lambda^*\leq c_6\left[ (\kappa_3 c_5)^2 + 1 \right] \cdot \|w^k-w^{k+1}\|^2,\qquad\forall w^{k+1}\in\mathbb{B}(\bar{w};\eta),
$$
ensuring that VP-EB is satisfied at $\bar{w}$.
\end{proof}

\end{document}